\documentclass{amsart}
\usepackage{amsmath,amsthm}
\usepackage{graphics}
\usepackage{color}
\usepackage{epsfig}
\usepackage{rlepsf}

\theoremstyle{plain}
\newtheorem{thm}{Theorem}[section]
\newtheorem{cor}[thm]{Corollary}

\newtheorem{lemma}[thm]{Lemma}
\newtheorem{claim}[thm]{Claim}

\newtheorem*{corbridgenumbers}{Corollary~\ref{cor:bridgenumbers}}
\newtheorem*{thmgeneralconstraints}{Theorem~\ref{thm:generalconstraints}}
\newtheorem*{coriff}{Corollary~\ref{cor:iff}}
\newtheorem*{thmteragaitobridgenumbers}{Theorem~\ref{thm:teragaitobridgenumbers}}

\theoremstyle{definition}
\newtheorem{defn}[thm]{Definition}
\newtheorem{remark}[thm]{Remark}

\newcommand{\hatF}{\ensuremath{{\widehat{F}}}}
\newcommand{\hatP}{\ensuremath{{\widehat{P}}}}
\newcommand{\hatR}{\ensuremath{{\widehat{R}}}}
\newcommand{\hatQ}{\ensuremath{{\widehat{Q}}}}
\newcommand{\hatA}{\ensuremath{{\widehat{A}}}}

\newcommand{\calL}{\ensuremath{{\mathcal L}}}
\newcommand{\calT}{\ensuremath{{\mathcal T}}}
\newcommand{\calC}{\ensuremath{{\mathcal C}}}

\newcommand{\comment}[1]{}

\newcommand{\bdry}{\ensuremath{\partial}}

\DeclareMathOperator{\nbhd}{\mathcal{N}}

\newcommand{\R}{\ensuremath{\mathbb{R}}}

\newcommand{\mobius}{M\"{o}bius }
\newcommand{\bn}{\ensuremath{{\mathfrak {b}}}}
\def\eop#1{\hfill\break\rightline{$\square$\ #1}}

\begin{document}

\title{Bridge number and integral Dehn surgery}
\author{K. L. Baker, C. McA. Gordon, and J. Luecke}



\begin{abstract} 
In a $3$-manifold $M$, let $K$ be a knot and $\hatR$ be an annulus which 
meets $K$ transversely.  We define the notion of the
pair $(\hatR,K)$ being caught by a surface $Q$ in the exterior of the
link $K \cup \partial \hatR$. For a caught pair $(\hatR,K)$, we consider 
the knot $K^n$ 
gotten by twisting $K$ $n$ times along $\hatR$  and give a lower bound on 
the bridge number of $K^n$ with respect to Heegaard splittings
of $M$ -- as a function of $n$, the genus of the splitting, and 
the catching surface $Q$. 
As a result, the bridge number of $K^n$ tends to infinity with $n$. 
In application, we look at a family
of knots $\{K^n\}$ found by Teragaito that live in a small Seifert 
fiber space $M$
and where each $K^n$ admits a Dehn surgery giving $S^3$. We 
show that the bridge number of $K^n$ with respect to
any genus $2$ Heegaard splitting of $M$ tends to infinity with $n$.
This contrasts with other work of the authors as well as with the 
conjectured picture for knots in lens spaces that admit Dehn surgeries 
giving $S^3$.
\end{abstract}
\subjclass{Primary 57M25 , 57M27}


\maketitle

\section{Introduction}
One may produce a family $\{K^n\}$ of knots in an orientable $3$--manifold $M$ by Dehn twisting a knot $K = K^0$ along an annulus $\hatR$ that it intersects transversely.  If $\hatR$ may be isotoped to lie in a genus $g$ Heegaard surface then $\bn_g(K^n)$, the minimal bridge number of $K^n$ among genus $g$ Heegaard splittings of $M$, is bounded. In this paper, we give sufficient
conditions to guarantee a converse to this statement. This allows us to give
examples of knots in the $3$-sphere whose integral Dehn surgeries contrast with
the results of \cite{bgl:bnaids} and \cite{bgl:og2hsfs} on non-integral
Dehn surgeries. When performing a Dehn surgery on a knot $K'$ in $S^3$, 
the core of the attached solid torus becomes a knot $K$ in the resulting 
manifold $M$. We refer to $K' \subset S^3$ and $K \subset M$ as 
{\em surgery-duals}. We show that there are knots in the $3$-sphere
for which an integral surgery is the same genus $2$ manifold $M$ and whose 
surgery-duals 
have unbounded bridge numbers with respect to any genus $2$ 
Heegaard splitting of $M$.

\begin{defn}\label{def:twistalongR}
Let $\hatR$ be an annulus embedded in $M$ with $\partial \hatR$ the link
$L_1 \cup L_2$ in $M$. Let $R=\hatR \cap (M - \nbhd(L_1 \cup L_2))$. 
Fix an orientation on $M$ and $\hatR$. This induces
an orientation on $L_i$ and its meridian $\mu_i$. 
Let $\hatR \times [0,1]$ be a product 
neighborhood of $\hatR$ in $M$ so that the corresponding interval orientation on $R \times [0,1]$ corresponds to the meridian orientation of $L_1$. Pick coordinates 
$\hatR=e^{2 \pi i \theta} \times [0,1]$, with $\theta \in [0,1]$, so that
$e^{2 \pi i \theta} \times \{0\}, \theta \in [0,1],$ is the oriented $L_1$. 
Define the homeomorphism
$f_n\colon  \hatR \times [0,1] \to \hatR \times [0,1]$ by $(e^{2 \pi i \theta},s,t) \to 
(e^{2 \pi i(\theta + nt)},s,t)$. Note that $f_n$ restricted to $\hatR \times \{0,1\}$
is the identity. Assume that the knot $K$ in $M$ intersects $\hatR \times
[0,1]$ in $[0,1]$ fibers. Let $K^n$ be the knot in $M$ gotten by applying 
$f_n$ to $K 
\cap (\hatR \times [0,1])$ (and the identity on $K$ outside this region). 
We refer to $K^n$ as {\em $K$ twisted $n$ times 
along $\hatR$}. Furthermore, note that $f_n$ induces a homeomorphism 
$h_n \colon M-\nbhd(L_1 \cup L_2) \to M-\nbhd(L_1 \cup L_2)$ by applying $f_n$
in $R \times [0,1]$ along with the identity outside this neighborhood.  We refer to this homeomorphism $h_n$ of $M-\nbhd(L_1 \cup L_2)$  as {\em $n$ Dehn-twists along the properly embedded annulus $R$}. 
Note that $K^n$ only depends on the isotopy class of $K$ in the complement of $L_1 \cup L_2$. Furthermore, one can check that $f_n,K^n,h_n$ are independent of 
the orientation chosen on $\hatR$. 
\end{defn}   

For an annulus $\hatR$ and knot $K$ in $M$, we say the pair $(\hatR, K)$ is {\em caught} 
if some oriented surface $Q$ properly embedded in the exterior $X = M-\nbhd(K \cup \bdry \hatR)$ intersects both components of $\bdry \nbhd (\partial \hatR)$ in slopes different than the framing induced by  $\hatR$ and with non-trivial homology on each of those components
(Definition~\ref{def:catches}). Lemma~\ref{lem:seifert} shows that it is often the case that
$(\hatR,K)$ is caught.

When $H_1 \cup_{\hatF} H_2$ is a Heegaard splitting of $M$ and $J$ is a knot in
$M$ we denote by $\bn_{\hatF}(J)$ the bridge number of $J$ with respect to this 
splitting (see section~\ref{sec:basics}). Here we allow Heegaard splittings in a manifold with boundary given by a 
union of compression bodies (see section~\ref{sec:spines}). The distance 
between two simple closed curves on a $2$-torus is the minimal geometric
intersection number of the curves up to isotopy (section~\ref{sec:slopes}). 
In section~\ref{sec:proofofmain}, we prove the following 

\begin{thm}\label{thm:generalconstraints}
Let $M$ be a compact, orientable $3$-manifold with (possibly empty) boundary and $K \cup L_1 \cup L_2$ a
link in $M$. Let $\hatR$ be an annulus in $M$ with 
$\partial \hatR= L_1 \cup L_2$. 
Assume $(\hatR,K)$ in $M$ is caught by the surface $Q$ in $X=M - \nbhd(K \cup L_1 \cup L_2)$.
Let $T_K,T_1,T_2$ be the components of $\partial X$ corresponding to $K,L_1,L_2$ respectively.
Fixing an orientation on $M$, let $\mu_i$ be a meridian of $L_i$ on $T_i$ and  $\lambda_i$ be a framing curve coming from $\hatR$. Express the first homology class of a component of $\bdry Q$ on $T_i$ as $p_i [\mu_i] + q_i[\lambda_i]$. Let
$\Delta_K$ be the distance on $T_K$ between a component of $\partial Q$ and a meridian of $K$ (setting $\Delta_K=0$ when
$Q$ is disjoint from $K$).
Let $K^n$ be $K$ twisted $n$ times along $\hatR$ (Definition~\ref{def:twistalongR}).  

Let $H_1 \cup_{\hatF} H_2$ be a genus $g$ Heegaard splitting of $M$.

Then either 
\begin{itemize}
\item[(1).]  $\hatR$ can be isotoped to lie in $\hatF$; or
\item[(2).] There is an essential annulus $A$ properly embedded in $X$ with a boundary component in each of $T_1$ and $T_2$. Furthermore, the slope of $\partial A$ on each $T_i$ is neither that of 
the meridian of $\nbhd(L_i)$ nor that of $\partial Q$; or
\item[(3).]  For each $n$, 
$$\bn_{\hatF}(K^n) \geq 
\frac{\min(|q_1+n p_1|,|q_2-n p_2|)/\max(-36\chi(Q),6) -2g + 1}{2 \max(\Delta_K,1)}$$

\end{itemize}
\end{thm}

To be able use the bound in conclusion $(3)$ of Theorem~\ref{thm:generalconstraints}, one needs to know that conclusions $(1)$ and $(2)$ do not hold. 
If $(1)$ holds then $\{\bn_{\hatF}(K^n)\}$ is a finite set. So assume $(1)$ does
not hold. If $(2)$ holds and $\partial A$ is not isotopic to 
$\partial \hatR$ on $T_1 \cup T_2$, then
$A$ can be used as a catching surface for $(\hatR,K)$. Applying 
Theorem~\ref{thm:generalconstraints} with $Q=A$ will force conclusion
$(3)$ or exhibit a new annulus annulus in $X$ whose boundary is isotopic
to $\partial \hatR$ on $T_1 \cup T_2$ (see 
Lemma~\ref{lem:boundaryslopesofannuli} and the proof of 
Corollary~\ref{cor:iff}). The first gives a lower bound in $n$ on 
$\bn_{\hatF}(K^n)$
in terms of the slopes of $\partial A$. 
On the other hand, if there is an annulus in $X$ whose boundary is isotopic
on $T_1 \cup T_2$ to $\partial \hatR$ then $\{\bn_g(K^n)\}$ 
will be finite (though not necessarily $\{\bn_{\hatF}(K^n)\}$), as described 
in the Corollary~\ref{cor:iff} below. 


Under the assumption that $(\hatR,K)$ is caught, the following gives necessary and sufficient conditions for the family of $\hatR$-twisted knots to have unbounded genus $g$ bridge numbers. 

\begin{cor}\label{cor:iff}
Let $M$ be a compact, orientable $3$-manifold with (possibly empty) boundary and $K \cup L_1 \cup L_2$ be a
link in $M$. Let $\hatR$ be an annulus in $M$ with $\partial \hatR= L_1 \cup L_2$, and let $R$ be the annulus $\hatR \cap (M - \nbhd(L_1 \cup L_2))$ properly
embedded in $M - \nbhd(L_1 \cup L_2)$.
Assume $(\hatR,K)$ in $M$ is caught by the surface $Q$ in $X=M - \nbhd(K \cup L_1 \cup L_2)$.
Let $T_K,T_1,T_2$ be the components of $\partial X$ corresponding to $K,L_1,L_2$ respectively. 
Fixing an orientation on $M$, let $K^n$ be $K$ twisted $n$ times along $\hatR$.

Assume $M$ has a genus $g$ Heegaard splitting.

Then either 
\begin{itemize}
\item[(1).] There is a genus $g$ Heegaard surface of $M$ containing $\hatR$; or
\item[(2).] There is an essential annulus $A$ in $X$ with one component of 
$\partial A$ on
$T_1$ and the other on $T_2$ such that $\partial A$ and $\partial R$ have the 
same slope on $T_1,T_2$; or
\item[(3).] $\bn_{g}(K^n) \to \infty$ as $n \to \infty$.
\end{itemize}
\medskip
Furthermore, if either $(1)$ or $(2)$ hold, then $\{\bn_{g}(K^n)\}$ is a finite 
set.
\end{cor}

Applying this to $M$ with small genus Heegaard splittings we have the following.

\begin{cor}\label{cor:bridgenumbers}
Assume $M$ closed and orientable and let $K \cup L_1 \cup L_2$ be a
link in $M$. Let $\hatR$ be an annulus in $M$ with 
$\partial \hatR= L_1 \cup L_2$, and let $R$ be the annulus $\hatR \cap (M - \nbhd(L_1 \cup L_2))$ properly
embedded in $M - \nbhd(L_1 \cup L_2)$. 
Assume that $(\hatR,K)$ in $M$ is caught. 
Assume there is no properly embedded, essential annulus $A$ in 
$X=M - \nbhd(K \cup L_1 \cup L_2)$ such that 
$\partial A \cap (T_1 \cup T_2)$ is
isotopic to $\partial R \cap (T_1 \cup T_2)$ on $T_1 \cup T_2$.
Fixing an orientation on $M$, let $K^n$ be $K$ twisted $n$ times along 
$\hatR$. 

If $M=S^3$ and $L_1 \cup L_2$ is not the trivial link, then  
$\bn_{0}(K^n) \to \infty$ as $n \to \infty$.

If $M$ is a lens space and $L_1 \cup L_2$ is not a lens space torus link, then 
$\bn_{1}(K^n) \to \infty$ as $n \to \infty$.

If $M$ has Heegaard genus at most $2$, then either 
$\bn_{2}(K^n) \to \infty$ as $n \to \infty$ or
one of the following holds:
\begin{itemize}
\item[(a).] $L_1$ has tunnel number $1$ in $M$ (or bounds a disk in $M$);
\item[(b).] $L_1$is a cable of a tunnel number $1$ knot in $M$ where the slope of the cabling annulus is that of $\bdry R$; or
\item[(c).] $0$--surgery (as framed by $R$) on $L_1$ contains an essential torus.
\end{itemize}
As $L_1$ and $L_2$ are isotopic in $M$, if any of $(a)-(c)$ holds for $L_1$, then it also holds
for $L_2$. 
\end{cor}

Beginning with an annulus $\hatR$ and banding $\bdry \hatR$ together in a sufficiently complicated manner, Osoinach produced infinite families of distinct knots in $S^3$ for which the same integral surgery produces the same manifold, 
$M$ \cite{osoinach}. The knots in such a family are related by twisting 
along $\hatR$, and the surgery-duals are related by twisting along an annulus
in $M$. Teragaito adapted this construction to develop an infinite family of distinct knots for which $+4$--surgery produces the same small Seifert fiber space
$M$, \cite{teragaito}.  In section~\ref{sec:teragaito}, we apply 
Corollary~\ref{cor:bridgenumbers} to prove the following
\begin{thm}\label{thm:teragaitobridgenumbers}
Let $\{K'^n\}$ be the Teragaito family of knots in $S^3$. For each $n$, let
$K^n \subset M$ be the $+4$-surgery-dual to $K'^n$ with respect to the Seifert 
framing on $K'^n$.  Then $\bn_0(K'^n) \to \infty$ and 
$\bn_2(K^n) \to \infty$ as $n \to \infty$.   
 \end{thm}
 
\begin{remark}
This is in sharp contrast to what occurs for non-integral surgeries.  
Corollary~1.1 of \cite{bgl:bnaids} shows that if a non-integral surgery on a hyperbolic knot in $S^3$ produces a small Seifert fiber space then the genus $2$ bridge number of the surgery-dual is at most $10975$. Theorem 2.4 of \cite{bgl:og2hsfs} shows that
if $p/q$-surgery, with $|q|>2$, on a hyperbolic knot in $S^3$ produces a manifold $M$
with Heegaard genus $2$, and $M$ contains no Dyck's surface, then the genus $2$
bridge number of the surgery-dual is at most $1$. To further
contrast the results of \cite{bgl:bnaids} and \cite{bgl:og2hsfs}, in 
section~\ref{sec:teragaito} we generalize the Teragaito family to
give other families of knots in the $3$-sphere, where each knot in a family 
admits a surgery giving the same genus $2$ manifold $M$ and where the 
surgery-duals to that family have arbitrarily large genus $2$ bridge numbers
in $M$ (Theorem~\ref{thm:hyperbolicexamples}). Generically these $M$ are hyperbolic manifolds, whereas for the Teragaito family $M$ is Seifert fibered.
In Lemma~\ref{lem:noDycks} we show that infinitely many of these hyperbolic
$M$ do not contain Dyck's surfaces, to support the contrast with Theorem 2.4
of \cite{bgl:og2hsfs}. 

\end{remark}

\begin{remark}
A conjecture of Berge says that if a knot $K'$ in $S^3$ admits a Dehn surgery
which is a lens space, $M$, then the bridge number of the surgery-dual 
$K \subset M$ with respect to a minimal 
genus Heegaard splitting of $M$ is one, i.e. $\bn_1(K)=1$.
Thus Theorem~\ref{thm:teragaitobridgenumbers} contrasts the expected picture
for lens space and small Seifert fiber space surgeries on knots in $S^3$. 
\end{remark}

\begin{remark}
Question~3.1 of \cite{mmm} asks if an integral surgery on a hyperbolic knot in $S^3$ produces a small Seifert fibered space $M$, then does the dual knot embed in a genus $2$ Heegaard surface for $M$. Teragaito showed that the dual knots to his 
examples answered this question in the negative -- that the dual knots do not
lie on a genus $2$ Heegaard surface. Theorem~\ref{thm:teragaitobridgenumbers} 
shows that in fact these knots have arbitrarily large bridge number
with respect to genus $2$ splittings of $M$.
\end{remark}

\begin{remark}
Teragaito also describes a related second infinite family of distinct knots for which $+4$--surgery always produces a certain small Seifert fibered space, \cite{teragaito}.  We conclude section~\ref{sec:teragaito} by showing that the set of genus $2$ bridge numbers of the knots surgery-dual to Teragaito's second family is bounded (Theorem~\ref{thm:secondfamily}).
\end{remark}

\subsection{Acknowledgements}
This work was partially supported by grant $\#209184$ to Kenneth L. Baker from
the Simons Foundation. The authors would like to thank Sean Bowman for
helpful conversations.

\section{Bounding Bridge Numbers}
\subsection{Slopes and surgeries}\label{sec:slopes}
A {\em slope} is an isotopy class of unoriented simple closed curves on a torus. We also say the slope of a collection of isotopic simple closed curves on a torus is the slope of any individual curve. The {\em distance} of two slopes $\alpha, \beta$ is the minimal geometric intersection number among curves representing these classes and is denoted $\Delta(\alpha,\beta)$.  Let $\mu$ be the meridianal slope of a knot $K$ in a manifold $M$.  Dehn surgery on a $K$ along a slope $\gamma$ is {\em integral} or {\em longitudinal} if $\Delta(\mu,\gamma)=1$, {\em non-integral} if $\Delta(\mu,\gamma)>1$, and {\em trivial} if $\Delta(\mu,\gamma)=0$.  In the surgered manifold, the core of the attached solid torus is the {\em surgery-dual} of $K$.

\subsection{Spines and core curves of handlebodies and compression bodies}\label{sec:spines}

A {\em spine} $\Gamma$ of a handlebody $H$ is a properly embedded graph such that $H-\Gamma \cong \bdry H \times (-\infty,0]$.  For a compression body $H$ with $\bdry H$ partitioned as $\bdry_+ H \cup \bdry_- H$ and $\bdry_+ H$ connected, then a {\em spine} $\Gamma$ of $H$ is a properly embedded graph  disjoint from $\bdry_+ H$ such that $H-(\Gamma \cup \bdry_- H) \cong \bdry_+ H \times (-\infty,0]$. 

An embedded closed curve $C$ in the interior of a handlebody or compression body $H$ is a {\em core curve} (or just {\em core}) if there is a spine $\Gamma$ of $H$ such that $C$ may be isotoped into $\Gamma$. Note that for a core $C$ of $H$, $H-\nbhd(C)$ is a compression body.  When $H$ is a solid torus, we usually speak of {\em the} core since all core curves are isotopic.

\subsection{Heegaard splittings, thin position, and bridge position}\label{sec:basics}
In this paper, a Heegaard splitting will always be a $2$--sided Heegaard 
splitting. In particular, a {\em Heegaard splitting} of a $3$--manifold with boundary,
$Y$, is the writing of $Y$ as the union of two compression bodies $H_1$ and $H_2$ along their boundary components $\bdry_+ H_1$ and $\bdry_+ H_2$. 
The shared boundary of these 
compression bodies is the {\em Heegaard surface} of the splitting. 
Given such a Heegaard surface 
$S$ of $Y$ there is a product $S \times \R \subset Y$ so that $S = S \times \{0\}$ and the complement of the product is the union of a pair of spines 
of the two compression bodies along with $\partial Y$.
This defines a height function on the complement in $Y$ of $\partial Y$ and  
the spines of the compression bodies.  Consider all the circles $C$ embedded in the product that are Morse with respect to the height function and represent the knot type of a particular knot $J$.  The following terms are all understood to be taken with respect to the Heegaard splitting.

Following \cite{gabai:fatto3mIII} (see also \cite{scharlemann}), the {\em width} of an embedded circle $C$ is the sum of the number of intersections $|C \cap S \times \{y_i\}|$ where one regular value $y_i$ is chosen between each pair of consecutive critical values of $C$.  The {\em width} of a knot $J$ is the minimum width of all such embeddings.  An embedding realizing the width of $J$ is a {\em thin position} of $J$, and $J$ is said to be {\em thin}.  

The minimal number of local maxima among Morse embeddings is the 
{\em bridge number} of $J$, and denoted $\bn_S(J)$, or, if $S$ is understood, 
$\bn(J)$.  
An embedding realizing the bridge number of $J$ 
may be ambiently isotoped so that all local maxima lie above all 
local minima, without introducing any more extrema.  
The resulting embedding is a {\em bridge position} of $J$, and 
$J$ is said to be {\em bridge}.
For a fixed genus $g$ of Heegaard splittings of $Y$, let $\bn_g(J)$
be the minimum bridge number of $J$ among genus $g$ Heegaard splittings of $Y$.

By definition, bridge numbers are positive. It is common to say that if $J$ can be isotoped
to lie on $S$ then $\bn_S(J)=0$. We will not use that terminology in this paper -- for such a
knot we take $\bn_S(J)=1$. That is, bridge and thin presentations
of a knot or link will always be Morse with respect to the given height function.

The definition of thin position extends to links.  If $K$ is a sublink of the link $J$, then a {\em $K$--thin position} of $J$ (with respect to the Heegaard splitting) is a thinnest (least width) position of $J$ among those that restrict to a thin position of $K$.

%
%
%
%
%

\subsection{$Q$ catches $(\hatR,K)$}\label{sec:lemmas}
Let $\hatR$ be an annulus embedded in the interior of an 
orientable $3$--manifold $M$ with 
$\bdry \hatR = L_1 \cup L_2$.   Let $K$ be a knot in $M$ disjoint 
from $L_1 \cup L_2$ and transverse to $\hatR$.
Write $\calL = K \cup L_1 \cup L_2$, let $X=M-\nbhd(\calL)$ be the exterior of the link $\calL$, and set $R=\hatR \cap (M - \nbhd(L_1 \cup L_2))$. Let
$T_K,T_1,T_2$ be the torus 
components of $\partial X$ corresponding to $K,L_1,L_2$ respectively.   

\begin{defn}\label{def:catches}
Let $Q$ be an oriented (possibly disconnected) surface, properly embedded in 
$X$ with no disk components or closed components. Furthermore, 
assume that if $Q$ has annular
components then $Q$ is a single annulus. We say that
$Q$ {\em catches} the pair $(\hatR,K)$ if 
\begin{itemize}
\item $\bdry Q \cap T_i$ is a non-empty collection of coherently oriented 
parallel curves on $T_i$ for each $i \in \{1,2\}$; and 
\item $\partial Q$ intersects both $T_1$ and $T_2$ in slopes different
than $\partial R$.  
\end{itemize}
\end{defn} 
We say the pair $(\hatR,K)$ is {\em caught} if it has a catching surface.

\begin{remark} Let $Q$ be a catching surface for $(K,\hatR)$. By discarding components, we 
may assume that each component
of $Q$ has some boundary component in $T_1 \cup T_2$. We may in fact assume 
that $Q$ has at most two components, and when $Q$ has two then one of these components is disjoint
from $T_1$ and the other disjoint from $T_2$. Note that if there were a disk
in $X$ with boundary on $T_1 \cup T_2$ then its boundary would have to be 
parallel to a component of the boundary of $\hatR$. If there were an annulus
in $X$ with only one boundary component on $T_1 \cup T_2$, then the existence
of $\hatR$ implies that the other must be on $T_K$. If there were two such
annuli, one with a boundary component on $T_1$, the other with a 
boundary component on $T_2$, these annuli could be used to construct a 
a single annulus with boundary in $T_1 \cup T_2$.  
\end{remark}

\begin{remark}
When $M$ is closed, the Half Lives, Half Dies Lemma says that the image of 
$\bdry_* \colon H_2(X,\bdry X) \to H_1(\bdry X)$ 
has half the rank of $H_1(\bdry X)$, e.g.\ Lemma~3.5 of \cite{hatcher:3mflds}.
This guarantees that there is a $Q$ such that
 \begin{itemize}
\item the components of $\bdry Q$ are coherently oriented parallel curves on the components of $\bdry X$ and
\item $[\bdry Q]$ is not a multiple of $[\bdry R]$ in $H_1(\bdry X)$
\end{itemize}

\end{remark}

\begin{defn}
Given a knot $K$ in a closed $3$--manifold $M$, we say an orientable surface $\Sigma$ with boundary that is properly embedded in $M-\nbhd(K)$ is a {\em generalized Seifert surface} for $K$ if $\bdry \Sigma$ is a collection of coherently oriented parallel curves on $\bdry \nbhd(K)$ once $\Sigma$ is oriented.  By the Half Lives, Half Dies Lemma, every such knot $K$ has a generalized Seifert surface.  Note that the boundary of a generalized Seifert surface may be a collection of meridianal curves.
\end{defn}

\begin{lemma}\label{lem:seifert}
A pair $(\hatR,K)$ in a closed 3-manifold $M$ is not caught if and only if $L_1$ has a generalized Seifert surface disjoint from $L_2$ and $K$ has a  generalized Seifert surface disjoint from either $L_1$ or $L_2$.
\end{lemma}

\begin{proof}

If $L_1$ does not have a generalized Seifert surface disjoint from $L_2$, then there exists one, say $\Sigma_1$, which is transverse to $L_2$ and $K$ such that when oriented $\Sigma_1 \cap \bdry \nbhd(L_2)$ is a non-empty collection of coherently oriented meridians of $L_2$.   Since the boundary slope of $\Sigma_1$ on $\bdry \nbhd(L_1)$ is necessarily different than that of $R$, $Q = \Sigma_1 \cap X$  catches $(\hatR,K)$.

If $K$ does not have a generalized Seifert surface disjoint from either $L_1$ or $L_2$, then there exists one, say $\Sigma_K$, which is transverse to $L_1 \cup L_2$ such that when oriented $\Sigma_K \cap \bdry \nbhd(L_i)$ is a non-empty collection of coherently oriented meridians of $L_i$, $i=1,2$.  Thus $Q = \Sigma_K \cap X$ catches $(\hatR,K)$.

\smallskip

Now assume $\Sigma_1$ is a generalized Seifert surface for $L_1$ that is disjoint from $L_2$ and transversely intersects $K$ and $\Sigma_K$ is a generalized Seifert surface for $K$ that is disjoint from $L_1 \cup L_2$.  Set $\Sigma_1' = \Sigma_1 \cap X$.  Recall that $\hatR$ is an annulus with boundary $L_1 \cup L_2$ that $K$ transversely intersects. Let $\lambda_2 = R \cap T_2$.


Observe that $[\bdry \Sigma_1']$, $[\bdry \Sigma_K]$, and $[\bdry R]$ together generate a rank $3$ subgroup of $H_1(\bdry X)$ whose intersection with $H_1(T_2)$ is generated by $[\lambda_2]$.  
If a surface $Q \subset X$ were to catch $(\hatR,K)$ then together $[Q \cap T_2]$ and $[\lambda_2]$ would generate a rank $2$ subgroup of $H_1(T_2)$.  But then $[\bdry Q]$ with $[\bdry \Sigma_1']$, $[\bdry \Sigma_K]$, and $[\bdry R]$ would generate a subgroup of $H_1(\bdry X)$ of rank at least $4$.  This contradicts the Half Lives, Half Dies Lemma.  Hence $(\hatR,K)$ cannot be caught.
\end{proof}
 
\begin{remark}
Notice that  $L_1$ has a generalized Seifert surface disjoint from $L_2$ if and only if its boundary slope on $\bdry \nbhd(L_1)$ agrees with the boundary slope of $R$.  If $\Sigma_1$ is a generalized Seifert surface for $L_1$ that is disjoint from $L_2$, then we may use copies of $\hatR$ to extend $\Sigma_1$ to a generalized Seifert surface for $L_2$ which an isotopy will make disjoint from $L_1$.  Hence $L_1$ has a generalized Seifert surface disjoint from $L_2$ if and only if $L_2$ has a generalized Seifert surface disjoint from $L_1$.
\end{remark}

\subsection{Combinatorics}
Let $\calL=K \cup L_1 \cup L_2$ be a link in a compact, orientable $3$-manifold
$M$ (possibly with boundary) and $X=M - \nbhd(\calL)$ its
exterior. Let $T_K,T_1,T_2$ be the torus components of $\partial X$ 
corresponding to $K,L_1,L_2$. In this section we assume that $Q$ is a properly
embedded, orientable surface in $X$ such that $\partial Q \cap T_i$ is non-empty and non-meridianal for $i \in \{1,2\}$.

Recall from section~\ref{sec:basics}, that $H_1 \cup_{\hatF} H_2$, 
a genus $g$ Heegaard splitting 
of $M$, gives a height function in the complement of the compression body 
spines and $\partial M$.
With $\calL$ in $K$--thin position with respect to this Heegaard splitting, isotop $Q$ so that in a neighborhood of any local maximum or minimum of $\calL$, $Q$ is below or above $\calL$ respectively, and $\bdry Q$ is transverse to the foliation by level curves on the rest of $T_1 \cup T_2 \cup T_K$ (if the components of $\partial Q \cap T_K$ are meridianal, we take each of these components to be level).  Near components of $\partial M$ we may take $Q$ transverse to the 
level surfaces. We take $Q$ transverse to the compression body spines. 
We may then further isotop $Q$, away from $\partial Q$ and away from the
compression body spines, to be transverse to the level surfaces except at a finite number of points which all occur at distinct levels, distinct from the extrema of $\calL$ too. 

Given any level surface $\hatP$ of this height function away from a critical level of $Q$, set $P = X \cap \hatP$.  By the above isotopy of $Q$, $\bdry Q$ intersects $\bdry P$ minimally on $\bdry X$. 
For such a level surface $\hatP$, form the corresponding pair of labeled {\em fat vertexed graphs of intersection} $G_P$ and $G_Q$, as follows. Define 
$\hatQ$ to be $Q$ with disks attached along the components of $\partial Q \cap (T_K \cup T_1 \cup T_2)$ when $\partial Q \cap T_K$ are not meridians of $K$, and
along the components of $\partial Q \cap (T_1 \cup T_2)$ when 
$\partial Q \cap T_K$ is meridianal.  Then $G_P$ and $G_Q$ are the graphs on the surfaces $\hatP$ and $\hatQ$, respectively, consisting of the {\em fat vertices} that are the disks $\nbhd(\calL) \cap \hatP$ on $\hatP$ and the disks that cap off $\partial Q$
in $\hatQ$, and the {\em edges} that are the arcs of $P \cap Q$. Note that
$\nbhd(K) \cap \hatP$ become vertices of valence $0$ when $\partial Q \cap
T_K$ is meridianal on $K$. {\em Label} the endpoint of an edge in one graph with the vertex of the other graph whose boundary contains the endpoint.

 Fix orientations of $\hatP$ and $\hatQ$.  Two vertices on the same graph and on the same component of $\bdry X$ are {\em parallel} if their corresponding oriented components of $\bdry P$ or $\bdry Q$ are coherently oriented on $\partial X$; they are {\em anti-parallel} otherwise.  The orientability of $P$, $Q$, and $X$ gives the {\em Parity Rule:} An edge connecting parallel vertices on one graph must connect anti-parallel vertices on the other graph.

Let $\Delta_K$, $\Delta_1$, $\Delta_2$ be the distance of slopes of $\bdry P$ and $\bdry Q$ on $T_K,T_1,T_2$.  If $\partial Q$ or $\partial P$ is disjoint from $T_K$, we
set $\Delta_K=0$. Note that if $\partial Q \cap T_K$ is meridianal on $K$, then $\Delta_K=0$.
By assumption, $\Delta_1,\Delta_2$ are non-zero. Set $m_K=|\hatP \cap K| \leq 2\bn_{\hatF}(K)$, $m_1 = | \hatP \cap L_1|$,  $m_2 = | \hatP \cap L_2|$; these are the numbers of vertices in $G_P$ corresponding to $K$, $L_1$, $L_2$ respectively. Number the components of $\partial P$ on a 
component of $\partial X$ in sequence $1, \dots, m_i$. 
 
Let $V_K$, $V_1$, $V_2$ be the sets of vertices of $G_Q$ corresponding to $K$, $L_1$, $L_2$ respectively ($V_K$ is empty when $\partial Q \cap T_K$ is meridianal). The vertices in each of these sets may also be numbered in the order they appear around their component of $\bdry X$.   Observe that a vertex $v \in V_i$ has valence $m_i \Delta_i$ for $i=K,1,2$; in particular, the $m_i$ labels of corresponding vertices in $G_P$ appear in order $\Delta_i$ times around $v$.  

Note that $P \cap Q$ has an arc component which is boundary parallel in $Q$
if and only if $G_Q$ has a monogon face, i.e.\ a face bounded by a fat vertex and single edge of $G_Q$.


\begin{lemma}\label{lem:levelPQ}
Given a Heegaard splitting $H_1 \cup_{\hatF} H_2$ of $M$,
there is a $K$--thin presentation for $\calL$  such that 
one of the following holds:
\begin{itemize}
\item[(A)] There is a level surface $\hatP$ transverse to $Q$ and with non-empty intersection with $L_1 \cup L_2$ such that there is no monogon of $G_Q$ at any vertex of $V_1 \cup V_2$. If the components of $\partial Q \cap T_K$ are meridianal on $K$, then $\hatP$ is disjoint from $\partial Q \cap T_K$.
\item[(B)] There is a level surface $\hatP$ transverse to $Q$ such that for some choice of $\{i,j\}=\{1,2\}$, $m_i \geq m_j=2$ and there is no monogon of $G_Q$ at any vertex of $V_i$.  If the components of $\partial Q \cap T_K$ are meridianal on $K$, then $\hatP$ is disjoint from $\partial Q \cap T_K$.
\item[(C)] $L_1 \cup L_2$ can be isotoped disjointly from $K$ (keeping $K$
fixed) so that $L_1$ and $L_2$ lie on
disjoint copies of $\hatF$.  
\end{itemize}

%
\end{lemma}

\begin{proof}
Take a $K$--thin presentation of $\calL$ with respect to the given splitting.
In this Morse presentation of $\calL$, let $I$ be a {\em middle slab}, i.e.\ an interval of level surfaces without critical points of $\calL$ in its interior whose upper and lower levels contain a maximum and minimum (resp.) of $\calL$.  

%
%
%

We choose $I$ so that the intersection of $L_1 \cup L_2$ with any level surface in $I$ is non-empty.   If there is a level surface $\hatP$ in $I$, transverse to $Q$, giving rise to no monogons in $G_Q$ at each of $V_1$ and $V_2$, then (A) is satisfied and 
we are done.  
(Possibly $\hatP$ is disjoint from $K$ or one, but not both, of $L_1, L_2$.)

So assume for each transverse level surface in this slab $I$ there is a high or low disk in $G_Q$ associated to $L_1 \cup L_2$. (A monogon of $G_Q$ is a high, resp.\ low, disk 
if a collar of its boundary lies above, resp.\ below, the level surface $\hatP$ in $M$.)  Apply Gabai's argument (in Lemma~4.4 of \cite{gabai:fatto3mIII}) to the high and low disks coming from these monogons of $G_Q$.  Note that near 
the maximum of $I$ such a disk must be high, and near the minimum it must be 
low. 
Gabai's argument in this context shows that there must be a level surface $\hatP$ that intersects some $L_j$ twice, for some $j \in \{1,2\}$, and gives rise to high and low disks in $G_Q$ guiding $L_j$ onto $\hatP$ disjointly from the other two components of $\calL$. Then $Q$ cannot also give rise to either a 
high or low disk at $\hatP$ for another component of $\calL$ since otherwise $\calL$ could be thinned without increasing the width of $K$.   Taking $\{i,j\}=\{1,2\}$, 
$\hatP$ satisfies (B) unless $\hatP$ is disjoint from $L_i$ --- which we now assume.

To the side of $\hatP$ containing $L_i$ we may find a new middle slab such that each level surface intersects $L_i$ but is disjoint from $L_j$.  Otherwise by isotoping $L_j$ onto $\hatP$ we could thin. Now apply the 
same argument.   Either we find a level surface satisfying (A) or $L_i$ can 
be isotoped disjointly from $K \cup L_j$ onto a level surface $\hatP'$ in this slab.  
Therefore, assuming (A) does not occur for a level surface in this new middle slab, we may isotop $L_1 \cup L_2$ disjointly from $K$ onto distinct level surfaces $\hatP$ and $\hatP'$ so that $L_1$ lies in one and $L_2$ in the other giving conclusion (C).
\end{proof}


\begin{lemma}\label{lem:largengivesnonhyp}
Let $M$ be an orientable, compact $3$-manifold and let $K \cup L_1 \cup L_2
\subset M$ be a link. Let $X = M - \nbhd(K \cup L_1 \cup L_2)$ and $T_K,T_1,T_2$ be the components of $\partial X$ corresponding to 
$K,L_1,L_2$.  Let $Q \subset X$ be a properly embedded, oriented surface
such that $T_i \cap \partial Q$ is a non-empty collection of coherently oriented
curves on $T_i$ for each $i \in \{1,2\}$. Let  $H_1 \cup_{\hatF} H_2$ be a 
genus $g$ Heegaard splitting of $M$. Assume that $L_1 \cup L_2$ cannot 
be isotoped
so that $L_1$ and $L_2$ lie on disjoint copies of $\hatF$. 
Let $\Delta_K$, $\Delta_1$, $\Delta_2$ be the distance between the slopes of $\partial Q$ and the 
meridian slopes of $K,L_1,L_2$ on the $T_K,T_1,T_2$. 
If  
\[\min(\Delta_1, \Delta_2) > \max(-36\chi(Q),6)(2\bn_{\hatF}(K)\max(\Delta_K,1)+2g-1)\] 
(where $\Delta_K=0$ includes the case that $\bdry Q$ is disjoint from $\bdry \nbhd(K)$)
then either 
\begin{itemize}
\item[(a).] there exists a \mobius band in
$X$ whose boundary is a meridian in $T_1$ or $T_2$; or
\item[(b).] there exists an annulus in $X$ with one boundary component essential 
on $T_1$ and the other essential on $T_2$. Furthermore, the slope of this annulus on $T_i$
is neither meridianal nor that of $\partial Q \cap T_i$, for each $i \in \{1,2\}$.
\end{itemize}
\end{lemma}

\begin{proof}
Recall that when $\Delta_K=0$, the components of $\partial Q \cap T_K$, if non-empty, are included
in the boundary of $\hatQ$ (the abstract surface in which $G_Q$ sits) and $V_K$ is empty. Also,
note that by convention $\bn_{\hatF}(K)>0$.

Applying Lemma~\ref{lem:levelPQ} to the given $\calL=K \cup L_1 \cup L_2$, $Q$, and Heegaard splitting, gives a level surface $\hatP$ of the splitting for which we assume that conclusion $(A)$ or $(B)$ holds. Then $m_i \geq m_j$ and, say, $i=1$ so that $G_Q$ has no monogons based at a vertex of $V_1$.

Let $G_Q(V_1)$ be the subgraph in a subsurface of $G_Q$ consisting of all edges of $G_Q$ that are incident to
$V_1$ and all the vertices of $G_Q$ to which these edges are incident. We think of $G_Q(V_1)$
as a graph in the surface gotten by attaching disks to $Q$ along those components of $\partial Q$ 
corresponding to vertices of $G_Q(V_1)$ (thus if a vertex of $G_Q$ is not connected by edges
to $V_1$, then it will give rise to a boundary component of $G_Q(V_1)$).
Let $\widetilde{G}_{Q}(V_1)$ be the reduced graph obtained from $G_{Q}(V_1)$ by amalgamating parallel edges.

\begin{claim}\label{claim:paralleledges1}
Assume $Q$ is an orientable surface with no disk components and such that
each component has non-empty boundary.
Furthermore assume that if $Q$ has annular components then $Q$ is a 
single annulus. Let $E$ be a collection of disjoint, properly embedded arcs 
in $Q$ such that no arc is parallel to the boundary and no two arcs are
parallel to each other. Then $|E| \leq \max(-3\chi(Q),1)$.

\end{claim}
\begin{proof}
If $Q$ is an annulus then $|E| \leq 1$, verifying the inequality. So assume
no component of $Q$ is an annulus.
Since no arc of $E$ is boundary parallel and no two are parallel, $E$ can be 
completed to an ideal triangulation of (the interior of) $Q$ by adding 
more edges between the components of $\bdry Q$ as needed.  If $E'$ is the
resulting collection of edges and $F$ is the collection ideal triangles, 
then we have both $3|F| = 2|E'|$ and $\chi(Q) = -|E'|+|F|$.  Thus $|E| \leq
|E'|=-3\chi(Q)$.  This gives the claim.
%
%
%
\end{proof} 


Since each vertex of $V_1$ has valence $m_1 \Delta_1$, 
Claim~\ref{claim:paralleledges1} shows that there must be at least 
$m_1 \Delta_1/\max(-6\chi(Q),1)$ mutually parallel edges of $G_{Q}(V_1)$.  
Let $\mathcal{E}$ be one of these sets of edges.  
%
%
%
Since the valence of a vertex of $V_K$ is $m_K \Delta_K \leq 2 \bn_{\hatF}(K) \Delta_K$ (the presentation is $K$--thin) which is in turn less than $m_1 \Delta_1/\max(-6\chi(Q),1)$ by hypothesis, 
 the edges in $\mathcal{E}$ cannot have an endpoint on a vertex of $V_K$ in $G_{Q}$.   
 Therefore 
 the edges in $\mathcal{E}$ either (a) join two vertices of $V_1$ (perhaps the same vertex) or (b) join a vertex of $V_1$ to a vertex of $V_2$ (note that this must
be the case if $Q$ is an annulus).

Now we show that there is a pair of edges of $\mathcal{E}$ bounding a disk on $\hatP-\nbhd(K)$.   
Let $G_P(\mathcal{E})$ be the subgraph of $G_P$ on $\hatP - \nbhd(K)$ consisting of the edges in $\mathcal{E}$ and the vertices from $\hatP \cap (L_1 \cup L_2)$ to which these edges are incident.  For case (a), these vertices are all the $m_1$ vertices of $\hatP \cap L_1$ ($\min(\Delta_1, \Delta_2) > \max(-6\chi(Q),1)$).  For case (b), notice that though the edges of $\mathcal{E}$ are parallel in
$G_{Q}(V_1)$, in $G_Q$ these edges may have monogons interspersed between them at the $V_2$ vertex. However, if there are such monogons then we are under conclusion $(B)$ of Lemma~\ref{lem:levelPQ}. Then $m_2=2$ and each of
the two vertices of $\hatP \cap L_2$ appears $|\mathcal{E}|/2$ times as a label at the $V_2$ end of 
$\mathcal{E}$. Whether we are working under conclusion $(A)$ or $(B)$ of Lemma~\ref{lem:levelPQ} then, the
hypotheses $\min(\Delta_1, \Delta_2) > \max(-6\chi(Q),1)$ and $m_1 \geq m_2$ tell us that in case (b), the vertices of $G_P(\mathcal{E})$ are all the $m_1$ vertices of $\hatP \cap L_1$ with all the $m_2$ vertices of $\hatP \cap L_2$. 
In both cases (a) and (b), the vertices of $G_P(\mathcal{E})$ have valence at least $\Delta_1/\max(-6\chi(Q),1)$.  (Each label of $G_{Q}$ at the vertices of $V_1$ or $V_2$ appears at least this many times at the endpoints of $\mathcal{E}$.  For (b) we use that $m_1 \geq m_2$.)

\begin{claim}\label{claim:paralleledges2}
Let $G$ be a graph in a surface $P$ with $\chi(P)=k$.  If $G$ has no monogons and each vertex has valence greater than $6 \max(1-k,1)$, then $G$ has parallel edges.
\end{claim}

\begin{proof}
Assume there are no parallel edges in $G$.
Then we may add edges to $G$ so that all faces are either $m$--gons with $m\geq 3$ or annuli with one boundary component being a component of $\bdry P$ and the other consisting of a single edge and vertex of $G$.  
We may then count $\chi(P)$ as $V-E+F = k$ where $V,E,F$ are the numbers of vertices, edges, and disk faces.  Because every edge is on the boundary of the faces (including the annuli) twice, $2E \geq 3F + |\bdry P|$.  Let $C=6 \max(1-k,1)$. The valence assumption implies $ CV < 2E$ and thus both that $V <2E/C$ and $C/2<E$. 

Therefore $k = V-E+F < 2E/C - E + 2E/3 - |\bdry P|/3$.  Hence 
$Ck < E (2-C/3)-|\partial P|C/3$. Then since $C \geq 6$, $k<0$. That is $C=6(1-k)$.
Thus $3(1-k) > (1-1/k) |\bdry P| + E \geq E$.  This contradicts that $C/2<E$.
\end{proof}

\begin{remark}
When $k>0$ or $|\partial P| \neq 0$, the above proof shows that if $G$ has no monogons and each
vertex has valence at least $6 \max(1-k,1)$, then $G$ has parallel edges. Change the strict
inequalities in the last four lines to $\leq, \geq$. We conclude that $k \leq 0$ and
$3(1-k) \geq (1-1/k) |\bdry P| + E > E$, the latter contradicting that $C/2 \leq E$. In the application
below, that $|\partial P|=0$ means that $K$ is disjoint from the level surface $\hatP$.
\end{remark}  

Observe that $G_P(\mathcal{E})$ has no monogons: in case (a) by the Parity Rule due to the coherency of orientations of $\bdry Q$ on the components of $\bdry \nbhd(L_1 \cup L_2) \subset \bdry X$, and in case (b) due to the endpoints of the edges being on vertices coming from different components of $L_1 \cup L_2$.  Note that in case (a) the 
vertices of $V_2$ are forgotten, so 
two edges that are parallel in $G_P(\mathcal{E})$ may not be parallel in $G_P$.
Also each vertex of $G_P(\mathcal{E})$ has valence at least $\min (\Delta_1, \Delta_2)/\max(-6\chi(Q),1) > 6(2\bn_{\hatF}(K)\max(\Delta_K,1)+2g-1) \geq 6\max(1-\chi(\hatP-\nbhd(K)),1)$ because $\hatP$ has genus $g$, $|\hatP \cap K| \leq 2\bn_{\hatF}(K)$.
Therefore Claim~\ref{claim:paralleledges2} implies that $G_P(\mathcal{E})$ has parallel edges.  
Hence there is a pair of edges $e,e' \in \mathcal{E}$ that bound a disk $D_{Q}$ in $G_{Q}$ and a disk $D_P$ in $G_P(\mathcal{E})$.   We may assume $D_{Q} \cap D_P = e \cup e'$.

In case (a), $D_Q \cup D_P$ is a \mobius band in $M - \nbhd(K \cup L_1)$ with boundary on $T_1$
that is a meridian.
This follows from the proof of Lemma~2.1 of \cite{gordon}. To see that the boundary is a meridian, one notes that its 
slope is the same as the slope of $\partial P$ since the rectangle $D_P$ connects anti-parallel
vertices in $P$ ($D_Q$ connects parallel vertices in $Q$).

In case (b), $D_Q \cup D_P$ is an annulus in $X$ with a boundary component on each of $T_1$ and $T_2$.   Each boundary component of this annulus must 
intersect a component of $\bdry P$ and of $\partial Q$ algebraically a non-zero number of times on $T_1 \cup T_2$. Thus a boundary component of this annulus is essential and isotopic to neither a component of $\bdry P$, a meridian, nor $\bdry Q$. This is conclusion $(b)$ of the Lemma.

To finish the proof we need to show that the \mobius band of case (a) can be taken to be disjoint from $L_2$. 


\begin{claim}\label{clm:S}
Either
\begin{itemize}
\item There is a \mobius band in $X$ whose boundary is a meridian on $T_1$ or $T_2$; or
\item There is an annulus in $X$ with a boundary component on 
each of $T_1$ and $T_2$ both of which are essential in $T_1,T_2$ and neither of which is 
isotopic to a 
meridian or to a component of $\partial Q$. 
\end{itemize}
\end{claim}

\begin{proof}
By the above, we may assume there is a \mobius band, $S$, in $M - \nbhd(K \cup L_1)$ with 
meridianal boundary in $T_1$. We assume there is no
such $S$ disjoint from $L_2$ and take $S$ to intersect $L_2$ minimally. Let $S'=S \cap X$. 
Isotop $\partial Q, \partial S'$ to intersect minimally in $\partial X$. Then no arc of $Q \cap S'$ is boundary parallel in $Q$ into $\partial Q \cap T_2$. Let $A$ be the
punctured annulus coming from the boundary of a regular neighborhood of $S'$ in $X$. Then no arc of $Q \cap A$ 
is boundary parallel in $Q$ into $\partial Q \cap T_2$ as there was no such for $Q \cap S'$. Consider the graphs of intersection $G_A,G_Q'$ coming from the arcs of $Q \cap A$ (as done for $G_P,G_Q$). Then $G_Q'$ has no monogons based at the vertices corresponding to $T_2$. The Parity Rule shows that $G_A$ has no monogons. We now apply the argument above to 
$G_A,G_Q'$ (in place of $G_P,G_Q$) to find a \mobius band, disk, or annulus in $X$. 

To fit that argument (despite the slight awkwardness of indices), set 
$V_1=|\partial Q \cap T_2|,V_2=|\partial Q \cap T_1|$ and $m_1=|\partial A \cap T_2| \geq 2$ and $m_2=|A \cap T_1|=2$. 
Then $m_1 \geq m_2$ and there are no monogons of $G_Q'$ at any vertex of $V_1$. This corresponds to
the situation in the above argument coming from conclusion $(B)$ of Lemma~\ref{lem:levelPQ} (with $A$
taking the role of $P$). Each vertex of $V_1$ in $G_Q'$ has 
valence $m_1 \Delta_2$. 
Let $G_Q'(V_1)$ be the subgraph of $G_Q'$ consisting of all edges of $G_Q'$ that are incident to
$V_1$ and all the vertices of $G_Q'$ to which these edges are incident. Again $G_Q'(V_1)$ is
a graph in the surface gotten by attaching disks to $Q$ along those components of $\partial Q$ 
corresponding to vertices of $G_Q'(V_1)$.
Let $\widetilde{G}_{Q}'(V_1)$ be the reduced graph obtained from $G_{Q}'(V_1)$ by amalgamating parallel edges. By Claim~\ref{claim:paralleledges1}, 
there must be at least $m_1 \Delta_2/\max(-6\chi(Q),1)$ mutually parallel edges of $G_{Q}'(V_1)$. Let $\mathcal{E}$ be one of these sets of edges.  Since $A$ is disjoint from $K$, 
 the edges in $\mathcal{E}$ either (a) join two vertices of $V_1$ (perhaps the same vertex) or (b) join a vertex of $V_1$ to a vertex of $V_2$. 
Let $G_A(\mathcal{E})$ be the subgraph of $G_A$ consisting of the edges in $\mathcal{E}$ and the vertices to which these edges are incident.  For case (a), these vertices are all the $m_1$ vertices 
corresponding to $A \cap T_2$ ($\Delta_2 > \max(-6\chi(Q),1)$). In this case we think
of $G_A(\mathcal{E})$ as a graph in the annulus $\hatA$ gotten by abstractly capping off the components of $A \cap T_2$ with disks (i.e. $V_2$ 
corresponds to the boundary of $\hatA$). For case (b), 
since $ \Delta_2 > \max(-6\chi(Q),1)$ and $m_1 \geq m_2$, the vertices of $G_A(\mathcal{E})$ are all the $m_1$ vertices corresponding to $A \cap T_2$ with both vertices of $A \cap T_1$. In case (b), we consider $G_A(\mathcal{E})$ as a 
graph in the $2$-sphere, $\hatA$, we get by abstractly capping off all of the boundary of 
$A$ with disks. In both cases (a) and (b), the vertices of $G_A(\mathcal{E})$ have valence at least $\Delta_2/\max(-6\chi(Q),1)$.  (Each label of $G_{Q}'$ at the vertices of $V_1$ or $V_2$ appears at least this many times at the endpoints of $\mathcal{E}$.  For (b) we use that $m_1 \geq m_2$.)

$G_A(\mathcal{E})$ has no monogons since $G_A$ has none. 
Also each vertex of $G_A(\mathcal{E})$ has valence at least $\min (\Delta_1, \Delta_2)/\max(-6\chi(Q),1) > 6(2\bn_{\hatF}(K)\max(\Delta_K,1)+2g-1) \geq 6$. 
Therefore Claim~\ref{claim:paralleledges2} (with $G=G_A(\mathcal{E})$ and $\hatA$ playing the role of $P$) implies that $G_A(\mathcal{E})$ has parallel edges.  
Hence there is a pair of edges $e,e' \in \mathcal{E}$ that bounds a disk $D_{Q}$ in $G_{Q}'$ and 
a disk $D_A$ in $G_A(\mathcal{E})$.   We may assume $D_{Q} \cap D_A = e \cup e'$.    

Then, as above, we have two possibilities. In case (a), $D_Q \cup D_A$ is a \mobius band in $X$ with boundary a meridian on $T_2$ (Lemma~2.1 of \cite{gordon}). 
In case (b), $D_Q \cup D_A$ is an annulus in $X$ with a boundary component on 
each of $T_1$ and $T_2$ both of which are essential in $T_1,T_2$ and neither of which is isotopic to a component of $\partial Q$ or $\partial A$. As $\partial A$ is meridianal on each of $T_1$ and $T_2$,
this completes the proof of the claim.

\end{proof}

Claim~\ref{clm:S} finishes the proof of Lemma~\ref{lem:largengivesnonhyp}.

\end{proof}

\subsection{Proof of main theorems}\label{sec:proofofmain}
In this section we prove Theorem~\ref{thm:generalconstraints}, Corollary~\ref{cor:iff}, and 
Corollary~\ref{cor:bridgenumbers}.

\begin{thmgeneralconstraints}
Let $M$ be a compact, orientable $3$-manifold with (possibly empty) boundary and $K \cup L_1 \cup L_2$ a
link in $M$. Let $\hatR$ be an annulus in $M$ with 
$\partial \hatR= L_1 \cup L_2$. 
Assume $(\hatR,K)$ in $M$ is caught by the surface $Q$ in $X=M - \nbhd(K \cup L_1 \cup L_2)$.
Let $T_K,T_1,T_2$ be the components of $\partial X$ corresponding to $K,L_1,L_2$ respectively.
Fixing an orientation on $M$, let $\mu_i$ be a meridian of $L_i$ on $T_i$ and  $\lambda_i$ be a framing curve coming from $\hatR$. Express the first homology class of a component of $\bdry Q$ on $T_i$ as $p_i [\mu_i] + q_i[\lambda_i]$. Let
$\Delta_K$ be the distance on $T_K$ between a component of $\partial Q$ and a meridian of $K$ (setting $\Delta_K=0$ when
$Q$ is disjoint from $K$).
Let $K^n$ be $K$ twisted $n$ times along $\hatR$ (Definition~\ref{def:twistalongR}).  

Let $H_1 \cup_{\hatF} H_2$ be a genus $g$ Heegaard splitting of $M$.

Then either 
\begin{itemize}
\item[(1).]  $\hatR$ can be isotoped to lie in $\hatF$; or
\item[(2).] There is an essential annulus $A$ properly embedded in $X$ with a boundary component in each of $T_1$ and $T_2$. Furthermore, the slope of $\partial A$ on each $T_i$ is neither that of 
the meridian of $\nbhd(L_i)$ nor that of $\partial Q$; or
\item[(3).]  For each $n$, 
$$\bn_{\hatF}(K^n) \geq 
\frac{\min(|q_1+n p_1|,|q_2-n p_2|)/\max(-36\chi(Q),6) -2g + 1}{2 \max(\Delta_K,1)}$$

\end{itemize}
\end{thmgeneralconstraints}

\leftline{\em Proof of Theorem~\ref{thm:generalconstraints}}

Let $H_1 \cup_{\hatF} H_2$ be the given genus $g$ Heegaard splitting of $M$.
Let $K, K^n, L_1, L_2, \hatR$, $X$, and $Q$ be as stated. Let $R$ be the
annulus $\hatR \cap (M - \nbhd(L_1 \cup L_2))$. If $R$ is compressible in
$M - \nbhd(L_1 \cup L_2)$, then $\hatR$ can be isotoped onto $\hatF$.
We hereafter assume that $R$ is incompressible. 




Dehn twists along the annulus $R$ provide homeomorphisms of $M-\nbhd(L_1 \cup L_2)$ in which the meridians of $L_1$ and $L_2$ are spun in opposite handedness around $\bdry \hatR$.  In particular, let $h_n\colon M - \nbhd(L_1 \cup L_2) \to M - \nbhd(L_1 \cup L_2)$ be the homeomorphism of Definition~\ref{def:twistalongR} obtained by twisting $n$ times
along $R$.
Define $\calL^n$ to be the link $K^n \cup L_1 \cup L_2$ and let $X_n$ be its
exterior in $M$. Then $h_n$  induces a homeomorphism $h_n' \colon X \to X_n$. 
Define $Q_n=h_n'(Q)$.

Use the meridian, longitude coordinates to express the first homology class of a component of $\bdry Q$ on $\bdry \nbhd(L_i)$ as $p_i [\mu_i] + q_i[\lambda_i]$. 
As $Q$ catches $(\hatR,K)$, $p_i \neq 0$.  
With these same coordinates, the first homology class of a component of $\bdry Q_n$ on $\bdry \nbhd(L_i)$ is $p_i[\mu_i] + (q_i + (-1)^{i+1} n p_i)[\lambda_i]$. In particular, the distance,  $\Delta_i^n$, between 
a component of $\bdry Q_n$ and the meridian $\mu_i$ on $\bdry \nbhd(L_i)$ is 
$|q_i + (-1)^{i+1} n p_i|$. Furthermore, the components
of $\partial Q_n$ are coherently oriented on $\partial \nbhd(L_i)$ since
those of $Q$ are.
In other words, $Q_n$ catches the pair $(\hatR,K^n)$ in $M$.

\begin{lemma}\label{lem:annulusbetween}
If $L_1 \cup L_2$ can  
be isotoped so that $L_1$ and $L_2$ lie on disjoint copies
of $\hatF$ of $M$. Then $\hatR$ can be isotoped to lie
in $\hatF$.
\end{lemma}

\begin{proof}
Isotop $L_1,L_2$ to lie in $\hatF_1,\hatF_2$, disjoint copies of $\hatF$. We may take $\hatF$ to lie between them. Isotop $\hatR$ so that it intersects $\hatF$
transversely. Then some curve, $c$, of $\hatR \cap \hatF$ will be a core curve of $\hatR$. $\hatR$
can be isotoped to a neighborhood of $c$ and then into $\hatF$.
\end{proof}

Thus we assume $L_1 \cup L_2$ cannot be isotoped so that $L_1,L_2$ lie on disjoint copies of $\hatF$.
We apply Lemma~\ref{lem:largengivesnonhyp} to $K^n,L_1,L_2,Q^n$. Note that conclusion $(a)$
cannot hold because of the annulus $R$ between in $M - \nbhd(L_1 \cup L_2)$ (e.g. Lemma~\ref{lem:boundaryslopesofannuli} below). If conclusion $(b)$
holds, then the annulus in $X$ must be essential in $X$ by the incompressibility of $R$ in 
$M - \nbhd(L_1 \cup L_2)$ (a compressing disk for the annulus in $X$ would give rise to one
for $R$). Thus 
conclusion $(b)$ gives conclusion $(2)$, and we may assume $(b)$ does not hold. Thus we must
conclude that 


\begin{align*}
\min(|q_1+n p_1|,|q_2-n p_2|) &= \min(\Delta_1^n, \Delta_2^n)  \\
               & \leq
\max(-36\chi(Q^n),6)(2\bn_{\hatF}(K^n)\max(\Delta_{K^n},1)+2g-1)
\end{align*}

As $\Delta_{K^n}=\Delta_K$ and $\chi(Q^n)=\chi(Q)$ we may rewrite this as

$$\bn_{\hatF}(K^n) \geq 
\frac{\min(|q_1+n p_1|,|q_2-n p_2|)/\max(-36\chi(Q),6) -2g + 1}{2 \max(\Delta_K,1)}$$

as desired.

\eop{(Theorem~\ref{thm:generalconstraints})}

We need the following for the proof of Corollary~\ref{cor:iff}.

\begin{lemma}\label{lem:boundaryslopesofannuli}
Let $N$ be an orientable $3$--manifold with toral boundary components $T_1,T_2$
( $\partial N$ may contain other components). 
Let $A$ be an incompressible annulus in $N$ with a boundary component on each of $T_1$ and $T_2$. Let $B$ be a $\bdry$--incompressible annulus or a \mobius band in $N$, in either case with essential boundary on $T_1 \cup T_2$. Then either 
\begin{itemize} 
\item each component of $\bdry B$ must be isotopic on $T_1 \cup T_2$ to
one of $\partial A$; or 
\item $\partial B$ has a component on each of $T_1$ and $T_2$ and $N$ is either 
$T^2 \times [0,1]$ or has $T^2 \times [0,1]$ as a
connected summand, where $T^2 \times \{0,1\}$ is $T_1 \cup T_2$.  
\end{itemize}
\end{lemma}
\begin{proof} Note since $A$ is incompressible, there is no essential disk in 
$N$ with boundary on $T_1 \cup T_2$. 

First, assume $\bdry B$ lies on $T_i$ and 
no component is isotopic to
$\partial A \cap T_i$. Isotop $\bdry B$ to intersect $\bdry A$ minimally on $T_i$. After possibly
surgering $B$ along trivial simple
closed curves of intersection with $A$, a disk in $A$ bounded
by an outermost arc of $A \cap B$ gives a $\bdry$--compressing
disk for $B$. Then $B$ must be a \mobius band and $\bdry$-compressing $B$ 
gives an essential disk in $N$ with boundary on $T_1 \cup T_2$, a contradiction.

So we assume that $\partial B$ has one component on $T_1$ and another
on $T_2$. Note that $B$ must be incompressible in $N$ (else there is an essential disk
at $T_1$ or $T_2$ in $N$).
Isotope $\partial B, \partial A$ on $T_1 \cup T_2$ to intersect
minimally.  Surger $A,B$ so that no closed curves of intersection are
trivial in either $A$ or $B$. By orientability (the Parity Rule), each arc of $A \cap B$
must connect different components of $\partial A$ and different components
of $\partial B$. Thus $A \cap B$ is a collection of parallel spanning
arcs in $A$ and in $B$. Take a pair that cobound a disk $D_1$ of $A - B$.
These arcs in $B$ then cobound a disk $D_2$ in $B$. Then $D_1 \cup D_2$
gives an annulus $C$ between $T_1 \cup T_2$ such that $\partial C$ can be isotoped to intersect $\partial B$ once on each of $T_1$ and $T_2$.
Indeed, we may isotop $C$ so that it intersects $B$ in a single arc.
Then $\nbhd(C \cup B \cup T_1 \cup T_2)$ has a $2$--sphere boundary
component that displays $N$ as a connected sum with $T_1 \times [0,1]$ as claimed.
\end{proof}

In terms of genus $g$ bridge numbers, Theorem~\ref{thm:generalconstraints} has a partial converse. 

\begin{coriff}
Let $M$ be a compact, orientable $3$-manifold with (possibly empty) boundary and $K \cup L_1 \cup L_2$ be a
link in $M$. Let $\hatR$ be an annulus in $M$ with $\partial \hatR= L_1 \cup L_2$, and let $R$ be the annulus $\hatR \cap (M - \nbhd(L_1 \cup L_2))$ properly
embedded in $M - \nbhd(L_1 \cup L_2)$.
Assume $(\hatR,K)$ in $M$ is caught by the surface $Q$ in $X=M - \nbhd(K \cup L_1 \cup L_2)$.
Let $T_K,T_1,T_2$ be the components of $\partial X$ corresponding to $K,L_1,L_2$ respectively. 
Fixing an orientation on $M$, let $K^n$ be $K$ twisted $n$ times along $\hatR$.

Assume $M$ has a genus $g$ Heegaard splitting.

Then either 
\begin{itemize}
\item[(1).] There is a genus $g$ Heegaard surface of $M$ containing $\hatR$; or
\item[(2).] There is an essential annulus $A$ in $X$ with one component of 
$\partial A$ on
$T_1$ and the other on $T_2$ such that $\partial A$ and $\partial R$ have the 
same slope on $T_1,T_2$; or
\item[(3).] $\bn_{g}(K^n) \to \infty$ as $n \to \infty$.
\end{itemize}
\medskip
Furthermore, if either $(1)$ or $(2)$ hold, then $\{\bn_{g}(K^n)\}$ is a finite
set.
\end{coriff}

\begin{proof}
Assume that conclusion $(3)$ above does not hold. Then there is genus $g$
Heegaard surface $\hatF$ of $M$ that fails inequality $(3)$ of 
Theorem~\ref{thm:generalconstraints}. Then 
Theorem~\ref{thm:generalconstraints} proves the Corollary unless there is an essential
annulus $A$ in $X$ with one component of $\partial A$ on $T_1$ and the other on $T_2$. We may also assume that $R$ is essential in $M - \nbhd(L_1 \cup L_2)$,
as otherwise conclusion $(1)$ of the Corollary will hold.
We show that $\partial A$ must have the same slopes as $\partial R$ on 
$T_1$ and $T_2$, giving conclusion $(2)$. Assume not. Then $\partial A$ must have different 
slopes on both
$T_1$ and $T_2$ from $\partial R$. Thus we may apply Theorem~\ref{thm:generalconstraints} 
using $A$ as the catching surface for $\hatR$. Again, this proves the Corollary
unless there is another essential annulus $A'$ in $X$ whose boundary has different
slopes on $T_1$ and $T_2$ from $\partial A$. Applying  
Lemma~\ref{lem:boundaryslopesofannuli} to $X$ shows that $X$ has $T^2 \times [0,1]$ as a
connected summand, where $T^2 \times \{0,1\}$ is $T_1,T_2$. Thus there is a $2$-sphere in $X$
separating $T_K$ from $T_1 \cup T_2$ and we may surger $R \cap X$ along this 2-sphere to 
obtained an essential annulus in $X$ with the same boundary as $R$, as desired.


We must show that if $(1)$ or $(2)$ hold, then $\{\bn_{g}(K^n)\}$ is a finite set.
Assume $(1)$ holds, and let $S$ be a genus $g$ Heegaard surface containing $\hatR$.
Now isotop $K$ keeping $\hatR$ fixed, so that it is bridge with respect to $S$.
Then $\{\bn_{g}(K^n)\}$ will be finite by the bridge number of this representative 
of $K$. 

Assume that $(2)$ holds. 
Let $M^n=M(-1/n,1/n)$ be the $-1/n,1/n$ Dehn surgeries on 
$L_1,L_2$ (respectively) in $M$ using the framings given by $R$. As in
Definition~\ref{def:twistalongR}, there is a homeomorphism of $M-\nbhd(L_1 \cup L_2)$ to itself,
$h_n$, that induces $h_n' \colon M^n \to M$. Furthermore, $h_n'$ identifies
the pair $(M^n,K)$ with the pair $(M,K^n)$. In the same way, twisting along $A$ induces a 
homeomorphism $f_n' \colon M^n \to M$ identifying the pair $(M^n,K)$ with $(M,K)$.
Thus $\bn_g(K^n)=\bn_g(K)$ for each $n$.



\end{proof}

\begin{remark}
Assume that conclusions $(1)$ and $(2)$ of Corollary~\ref{cor:iff} 
do not hold. The proof of Corollary~\ref{cor:iff} shows that
either

\begin{itemize}
\item[(A).] For each $n$, 

$$\bn_g(K^n) \geq 
\frac{\min(|q_1+n p_1|,|q_2-n p_2|)/\max(-36\chi(Q),6) -2g + 1}{2 \max(\Delta_K,1)}$$
\smallskip
where $(p_i,q_i)$ are the coordinates of $\partial Q$ on $T_i$
(framed by $R$ as above) and $\Delta_K$ is the distance on $T_K$ between a 
component of $\partial Q$ and a meridian of $K$ (setting $\Delta_K=0$ when
$Q$ is disjoint from $K$); or
\item[(B).] There is an annular catching surface $Q'$ for $\hatR$ in $M$. Let
$(r_i,s_i)$ be the coordinates of $\partial Q'$ on $T_i$ (framed by $R$). Then
for each $n$, 

$$\bn_g(K^n) \geq \min(|s_1+n r_1|,|s_2-n r_2|)/12 -g + 1/2$$

\end{itemize}

\end{remark}

We finish with the proof of the following. 

\begin{corbridgenumbers}
Assume $M$ closed and orientable and let $K \cup L_1 \cup L_2$ be a
link in $M$. Let $\hatR$ be an annulus in $M$ with 
$\partial \hatR= L_1 \cup L_2$, and let $R$ be the annulus $\hatR \cap (M - \nbhd(L_1 \cup L_2))$ properly
embedded in $M - \nbhd(L_1 \cup L_2)$. 
Assume that $(\hatR,K)$ in $M$ is caught. Let $X=M - \nbhd(K \cup L_1 \cup L_2)$ and $T_1,T_2$
be the components of $\partial X$ coming from $L_1,L_2$.
Assume there is no properly embedded, essential annulus $A$ in 
$X$ such that 
$\partial A \cap (T_1 \cup T_2)$ is
isotopic to $\partial R \cap (T_1 \cup T_2)$ on $T_1 \cup T_2$.
Fixing an orientation on $M$, let $K^n$ be $K$ twisted $n$ times along 
$\hatR$. 

If $M=S^3$ and $L_1 \cup L_2$ is not the trivial link, then  
$\bn_{0}(K^n) \to \infty$ as $n \to \infty$.

If $M$ is a lens space and $L_1 \cup L_2$ is not a lens space torus link, then 
$\bn_{1}(K^n) \to \infty$ as $n \to \infty$.

If $M$ has Heegaard genus at most $2$, then either 
$\bn_{2}(K^n) \to \infty$ as $n \to \infty$ or
one of the following holds:
\begin{itemize}
\item[(a).] $L_1$ has tunnel number $1$ in $M$ (or bounds a disk in $M$);
\item[(b).] $L_1$ is a cable of a tunnel number $1$ knot in $M$ where the slope of the cabling annulus is that of $\bdry R \cap T_1$; or
\item[(c).] $0$--surgery (as framed by $R$) on $L_1$ contains an essential torus.
\end{itemize}
As $L_1$ and $L_2$ are isotopic in $M$, if any of $(a)-(c)$ holds for $L_1$, then it also hold
for $L_2$. 
\end{corbridgenumbers}

\begin{proof}
Under the hypotheses given, Corollary~\ref{cor:iff} implies that if $\bn_g(K^n)$ does not tend to infinity with $n$ then $\hatR$ lies on a genus $g$ Heegaard splitting of $M$,
$H_1 \cup_{\hatF} H_2$.
The conclusions for $g=0$ and $g=1$ are then immediate. 
So assume $g=2$.   

If $F=\hatF - \nbhd(L_1)$ is compressible in the complement of $L_1$, then such a compression shows that $L_1$ is a cable of a core of either $H_1$ or $H_2$.  In this case either $L_1$ has tunnel number $1$ or is the cable of a tunnel
number one knot.   If on the other hand $F$ is incompressible, then the Handle Addition Lemma (Lemma~2.1.1 of \cite{cgls:dsok}) implies that surgery on $L_1$ along the slope induced by $F$ is toroidal.  
\end{proof}

\section{Application to Teragaito's Example and Some Generalizations}\label{sec:teragaito}
Osoinach describes a construction producing infinitely many distinct knots in $S^3$ (or some other manifold) for which the same integral surgery on each knot yields the same new manifold $M$, \cite{osoinach}.  Dually, this may be viewed as infinitely many distinct knots in a manifold $M$ (that is, no homeomorphism of $M$ takes one knot to the other) for which the same integral surgery yields $S^3$.     Teragaito gives a specific example of this construction in which the manifold $M$ is a small Seifert fiber space \cite{teragaito}.  We produce a two-parameter generalization of Teragaito's examples in which the resulting manifolds $M$ have Heegaard genus two and are typically hyperbolic.   We apply Corollary~\ref{cor:bridgenumbers} to show that Teragaito's family of knots and 
our generalizations (for large parameter values) have genus $2$ bridge numbers 
in $M$ that tend to infinity.  Let us first overview Teragaito's example.

Teragaito describes a $3$--component link $\calL' = K'\cup L_1 \cup L_2$ in $S^3$ where $L_1 \cup L_2$ is the boundary of an annulus $A$ and there is a pair of pants $R$ (that intersects the interior of $A$) expressing $K'$ as a banding of $L_1 \cup L_2$ and meeting $L_1 \cup L_2$ with the same framing as $A$ as shown in Figure~\ref{fig:teragaitolink}.  Frame the components of the link $\calL'$ with $R$.  Then, as Teragaito shows, $0$--surgery on $K'$ (that is, a $+4$--surgery with respect to the Seifert framing) produces a small Seifert fiber space $M$ containing the knot $K$ dual to the surgery and an annulus $\hatR$ with boundary $L_1 \cup L_2$.   The annulus $\hatR$ is obtained after surgery by capping off the $K'$ component of $\bdry R$ with a disk.  Indeed the interior of $\hatR$ is pierced once by $K$ in $M$.

\begin{figure}
\includegraphics[width=5in]{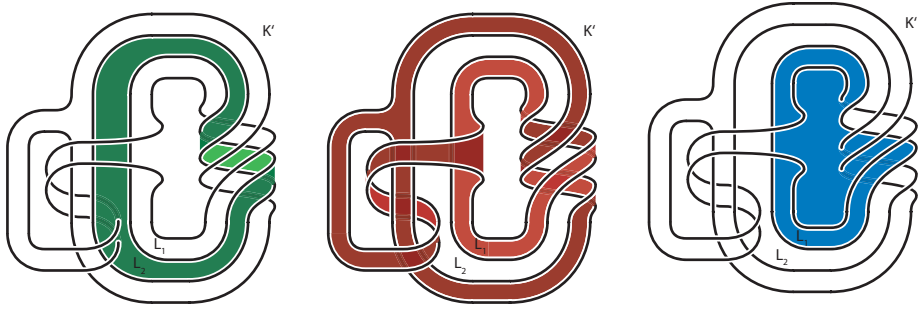}
\caption{The link $\calL' =K' \cup L_1 \cup L_2$ shown with the green annulus $A$, the red pair of pants $R$, and the blue disk $\hatQ$.}
\label{fig:teragaitolink}
\end{figure}


We generalize Teragaito's example by inserting extra twists in two regions.  Figure~\ref{fig:teragaitolinkwtwists} is the same as Figure~\ref{fig:teragaitolink} except that two (unlinked) unknots $J_0$ and $J_1$ have been added and the full twist on the right hand side of the link $\calL'$ has been undone which may be restored by a $-1$--surgery along $J_1$.   Produce the link $\calL'_{j_0,j_1} = K'_{j_0,j_1} \cup (L_1)_{j_0,j_1} \cup (L_2)_{j_0,j_1}$ in $S^3_{j_0,j_1} \cong S^3$ by performing $-1/j_0$--surgery on $J_0$ and $-1/j_1$--surgery on $J_1$.  The link $\calL'_{0,1} = \calL'$ is the link used in Teragaito's example.  As one may conclude from Figure~\ref{fig:teragaitolinkwtwists}, $(L_1)_{j_0,j_1} \cup (L_2)_{j_0,j_1}$ cobound a green annulus $A_{j_0,j_1}$ and there is a red pair of pants $R_{j_0,j_1}$ (intersecting the interior of $A_{j_0,j_1}$) expressing $K'_{j_0,j_1}$ as a banding of $(L_1)_{j_0,j_1} \cup (L_2)_{j_0,j_1}$ and meeting $(L_1)_{j_0,j_1} \cup (L_2)_{j_0,j_1}$ in the same framing as $A_{j_0,j_1}$.  The component $J_0$ links the banding so that $-1/j_0$--surgery on $J_0$ inserts $j_0$ full twists into the band.  

Frame the components of the link $\calL'_{j_0,j_1}$ with $R_{j_0,j_1}$.  (Observe that each of the link components of Figure~\ref{fig:teragaitolinkwtwists} is an unknot and the framing induced by $R$ is the standard Seifert framing.  Twisting along $J_0$ and $J_1$ will twist these framings.) Then $0$--surgery on $K'_{j_0,j_1}$ produces a manifold $M_{j_0,j_1}$ containing the knot $K_{j_0,j_1}$ dual to the surgery and an annulus $\hatR_{j_0,j_1} \subset M_{j_0,j_1}$ with boundary $(L_1)_{j_0,j_1} \cup (L_2)_{j_0,j_1}$.   The annulus $\hatR_{j_0,j_1}$ is obtained after surgery by capping off the $K'_{j_0,j_1}$ component of $\bdry R_{j_0,j_1}$; the interior of $\hatR_{j_0,j_1}$ is pierced once by $K_{j_0,j_1}$ in $M_{j_0,j_1}$.

As $(L_1)_{j_0,j_1}$ is an unknot in $S^3$, it bounds a disk $\hatQ_{j_0,j_1}$.  This disk is punctured $2|j_1|$ times by $K'_{j_0,j_1}$ and $|j_1|$ times by $(L_2)_{j_0,j_1}$. Let $X_{j_0,j_1}$ be the exterior of the link $\calL'_{j_0,j_1}$ in $S^3$.  
Let $Q_{j_0,j_1} = \hatQ_{j_0,j_1} \cap X_{j_0,j_1}$ be this $3|j_1|$--punctured disk properly embedded in 
$X_{j_0,j_1}$ suggested in blue by the right hand picture in Figure~\ref{fig:teragaitolinkwtwists}.  The blue $3$--punctured disk $Q_{0,1}$ is shown in Figure~\ref{fig:teragaitolink}.

Let us now drop the subscripts $j_0,j_1$ from our notation except when needed.
Thus hereafter $K',L_1,L_2, A, R, M, X$ correspond to those with subscripts $j_0,j_1$. 

$A$ is an annulus in $S^3$ with $\partial A = L_1 \cup L_2$ and $\hatR$ is an
annulus in $M$ with $\partial \hatR = L_1 \cup L_2$. Twisting $K'$ along $A$
produces the family of knots $\{K'^n\}$ in $S^3$ and twisting $K$ 
along $\hatR$ produces the family $\{K^n\}$ in $M$. 
Let $M_n$ ($S^3_n$) be the manifold obtained from $M$ ($S^3$, resp.) by $-1/n$--surgery on $L_1$ and $1/n$--surgery on $L_2$. In both $M_n$ and $S^3_n$ we 
continue to use the names $L_1$ and $L_2$ for the knots dual to these Dehn surgeries. As in Definition~\ref{def:twistalongR}, there are homemorphisms identifying
the pair $(M_n,K \cup L_1 \cup L_2)$ with $(M,K^n \cup L_1 \cup L_2)$ and the
pair $(S^3_n,K' \cup L_1 \cup L_2)$ with $(S^3,K'^n \cup L_1 \cup L_2)$.
Use the framing on $K'$ (by $R$) and the identification $(S^3_n,K') \cong 
(S^3,K'^n)$ to assign a framing to $K'^n$. Then the knot dual to the 
$0$--surgery on $K'^n$ in $S^3$ is the knot dual to the $0$--surgery on $K'$
in $S^3_n$, which by definition is $K$ in $M_n$. But this is identified with
$K^n$ in $M$. That is, we see that
$K^n$ is the dual knot to the $0$--surgery on $K'^n$. 
Finally, observe that 
$X = S^3 - \nbhd(K' \cup L_1 \cup L_2) \cong S^3 - \nbhd(K'^n \cup L_1 \cup L_2) \cong M-\nbhd(K^n \cup L_1 \cup L_2) \cong M - \nbhd(K \cup L_1 \cup L_2)$.

\begin{figure}
\includegraphics[width=5in]{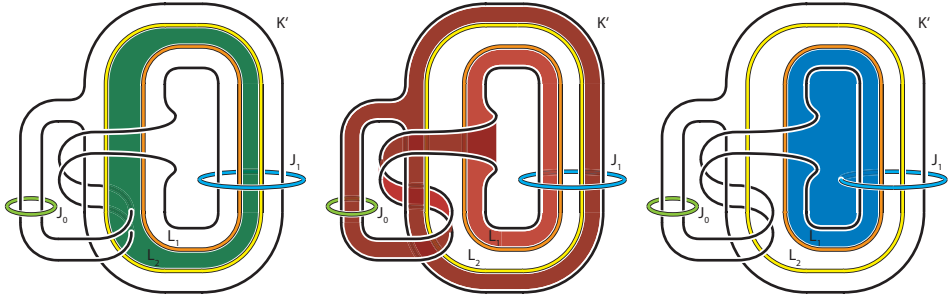}
\caption{The link $\calL' =K' \cup L_1 \cup L_2$ shown with the green annulus $A$, the red pair of pants $R$, and the blue disk $\hatQ$.}
\label{fig:teragaitolinkwtwists}
\end{figure}


%

\begin{thmteragaitobridgenumbers}
Let $\{K'^n\}$ be the Teragaito family of knots in $S^3$. For each $n$, let
$K^n \subset M$ be the $+4$-surgery-dual to $K'^n$ with respect to the Seifert 
framing on $K'^n$.  Then $\bn_0(K'^n) \to \infty$ and 
$\bn_2(K^n) \to \infty$ as $n \to \infty$. 
\end{thmteragaitobridgenumbers}

\begin{proof}
%
Recall that the Teragaito family is where $j_0=0,j_1=1$ and the $+4$-surgery
in the Seifert framing is our $0$-surgery when framed by $R$.
The $3$--punctured disk $Q$ in the exterior $X$ obtained from the disk $\hatQ$ has 
one component of its boundary on $\bdry \nbhd(L_1)$ and one component of its 
boundary on $\bdry \nbhd(L_2)$.  As $j_1 \neq 0$ these slopes both differ from 
the slopes of $\bdry \hatR$,  and so $Q$ catches $(\hatR, K)$ and $(A,K')$.  



Teragaito shows the link exterior, $X$, of $K' \cup L_1 \cup L_2$ in $S^3$ (and 
of $K \cup L_1 \cup L_2$ in $M$) is hyperbolic; hence, in particular
$X$ contains no essential annulus. We apply Corollary~\ref{cor:bridgenumbers}. As $L_1 \cup L_2$ is
not trivial in $S^3$, $\bn_0(K'^n) \to \infty$ as $n \to \infty$. By 
Lemma~\ref{lem:verify}, $\bn_2(K^n) \to \infty$ as $n \to \infty$.
\end{proof}

The Teragaito family $\{K^n\}$ is thus a family of knots in the Seifert fiber space $M$ of 
unbounded bridge number each of which nevertheless admits an $S^3$ surgery. We show that the
above generalization yields such families of knots (arbitrarily large genus 2 bridge number
where each knot admits an $S^3$ surgery) in manifolds $M$ which are hyperbolic.

\begin{defn}\label{def:kappa}
{\tt SnapPy} \cite{snappy} shows that the manifold $W = S^3-\nbhd(K' \cup J_0 \cup J_1)$ of Figure~\ref{fig:teragaitolinkwtwists} is hyperbolic.  It also verifies that $W_0$, the Dehn filling of $W$ along the slope of $\bdry R$ (i.e.\ slope $0$) on the component of $\bdry W$ coming from $\bdry \nbhd(K')$, is hyperbolic. {\tt SnapPy} also shows that $Y = S^3-\nbhd(K' \cup J_0 \cup J_1 \cup L_2)$ is hyperbolic.
By Thurston's Hyperbolic Dehn Surgery Theorem, there is a constant $\nu$, which we will take to be greater than $2$, such that as long as $\min\{|j_0|,|j_1|\} \geq 
\nu$ then
\begin{itemize}
\item $M$, which is the Dehn filling of $W_0$ along the slopes 
$-1/j_0$ and $-1/j_1$ on the components of $W_0$ coming from $\bdry \nbhd(J_0)$ and $\bdry \nbhd(J_1)$ respectively, is hyperbolic.
\item $Y_{j_0}$, the $-1/j_0$-Dehn filling of $Y$ along the component of
$\partial Y$ coming from $\bdry \nbhd(J_0)$, is hyperbolic. 
\end{itemize}
\end{defn}


\begin{thm}\label{thm:hyperbolicexamples}
For $|j_0|,|j_1| \geq \nu$, $M$ has Heegaard genus $2$ and $\bn_2(K^n) \to \infty$ as $n \to \infty$.
\end{thm}

\begin{proof}
This follows from Corollary~\ref{cor:bridgenumbers}, 
Lemma~\ref{lem:noannulus}, and Lemma~\ref{lem:verify}.
\end{proof}

\begin{lemma}\label{lem:noannulus}
Let $X=M - \nbhd(K^n \cup L_1 \cup L_2)$ and $T_1,T_2$ the components of 
$\partial X$ corresponding to $\partial \nbhd(L_1), \partial \nbhd(L_2)$. 
If $|j_0| \geq \nu, j_1 \neq 0$, there is no essential annulus in $X$ with one
boundary component on $T_1$ and the other on $T_2$. 
\end{lemma}

\begin{proof}
As $L_1$ is isotopic to a meridian of $J_1$ in Figure~\ref{fig:teragaitolinkwtwists}, we can write $X = Y_{j_0} \cup_{T} C(|j_1|,r)$ where $Y_{j_0}$  is 
as in Definition~\ref{def:kappa}, $C(|j_1|,r)$ is cable space between $T$ and
$T_1$, and $T$ corresponds to $\partial \nbhd(J_1)$ in $\partial Y_{j_0}$.  
An essential annulus in $X$ with one boundary component on $T_1$ and the 
other on $T_2$, would give rise to an essential annulus in $Y_{j_0}$ 
with boundary on $T$ and $T_2$ ($T$ is incompressible in $Y_{j_0}$).
But this contradicts the hyperbolicity of $Y_{j_0}$.
\end{proof}

In support of the above theorems we proceed to show 

\begin{lemma}\label{lem:verify}
For any $j_0,j_1$, $M$ has Heegaard genus $2$. Furthermore, if either $(a)$
$j_1=\pm1$ and $|j_0| \neq 1,2$  or $(b)$ $|j_0|, |j_1|\geq \nu \geq 3$ then for each $i \in \{1,2\}$ the link component $L_i \subset M$
\begin{itemize} 
\item  has tunnel number greater than $1$,
\item is not a cable of a tunnel number one knot where the slope of the cabling 
annulus is that of $\partial \hatR$, and 
\item has an atoroidal Dehn surgery along the slope $\bdry \hatR$.  
\end{itemize}
\end{lemma}
 

\begin{figure}
\includegraphics[width=5in]{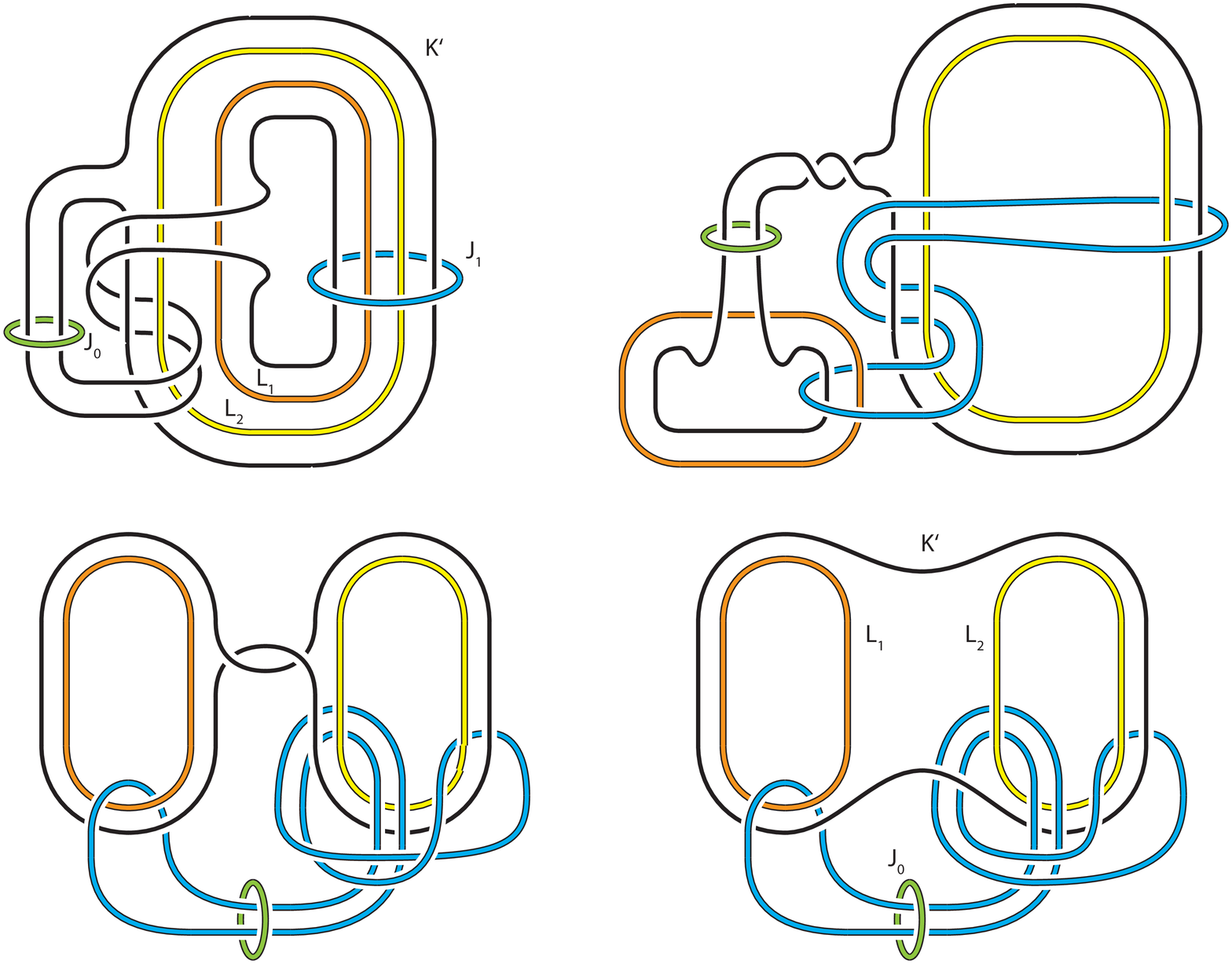}
\caption{}
\label{fig:simplifyingisotopy}
\end{figure}

\begin{figure}
\includegraphics[width=5in]{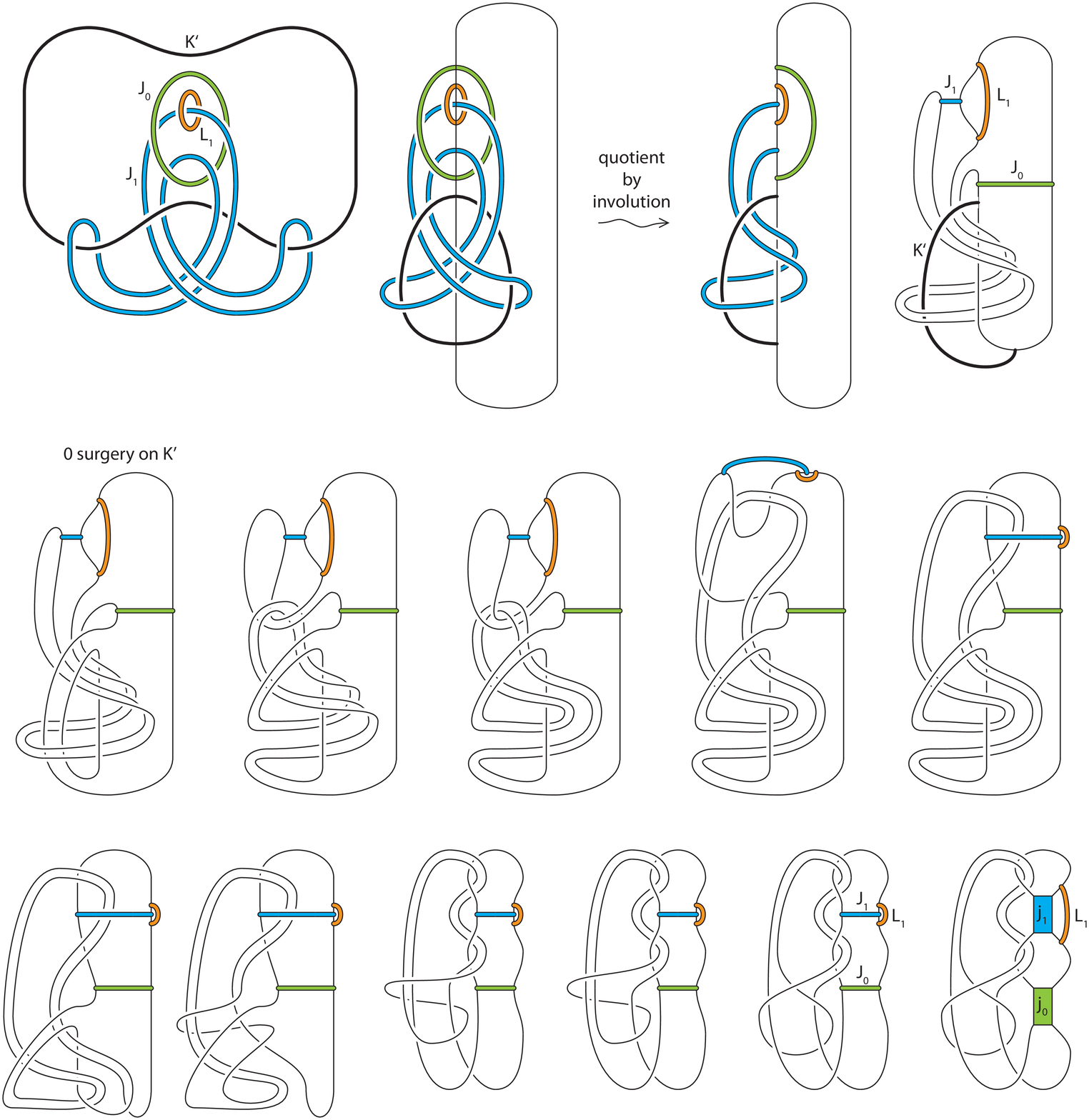}
\caption{}
\label{fig:tangleisotopy}
\end{figure}

\begin{proof}

We assume that $i=1$. Then the annulus $\hatR$ shows that the same statements
hold for $L_2$.

{\bf (1) $M$ has Heegaard genus $2$.}

To start, Figure~\ref{fig:simplifyingisotopy} shows an isotopy of 
$K' \cup L_1 \cup L_2 \cup J_0 \cup J_1$ into a simplified configuration.
After dropping $L_2$, performing a further isotopy makes the remaining link $K' \cup L_1 \cup J_0 \cup J_1$ strongly invertible as shown in the first picture of Figure~\ref{fig:tangleisotopy}.  The second picture continues the isotopy and exhibits the fixed set of this strong inversion.  Each component of the link is an unknot bounding a disk that is also invariant under the involution. The framings of these disks agree with their page framings (i.e.\ blackboard framings).  The third picture shows the quotient of the involution.  Each link component except $J_1$ projects to an arc with the page framing. The fourth picture shows an isotopy of the arc $J_1$ and the fixed set that restores its framing to the page framing.

The second row of Figure~\ref{fig:tangleisotopy} begins with a banding along (the arc corresponding to) $K'$ using its page framing.  By the Montesinos trick, this is equivalent to doing $0$--surgery on $K'$.  The remaining sequence of pictures of Figure~\ref{fig:tangleisotopy} up to the penultimate one exhibit isotopies of the fixed set and the arcs corresponding to $L_1 \cup J_0 \cup J_1$.  Throughout these isotopies the page framings of the arcs are unaltered. 

The final picture of Figure~\ref{fig:tangleisotopy} replaces the horizontal arcs $J_0$ and $J_1$ with rectangles indicating vertical runs of $|j_0|$ and $|j_1|$ half-twists.  The signs of $j_0$ and $j_1$ dictate the handedness of the twists as illustrated in Figure~\ref{fig:twistboxes}.  These replacements correspond to performing $-1/j_0$ and $-1/j_1$ surgeries on $J_0$ and $J_1$ in the double branched cover.  The double branched cover of the resulting link $\ell$ is the manifold $M=M_{j_0,j_1}$.   Since the link $\ell$ is $3$--bridge, $M$ is a manifold of Heegaard genus at most $2$.

\begin{figure}
\includegraphics[width=2.5in]{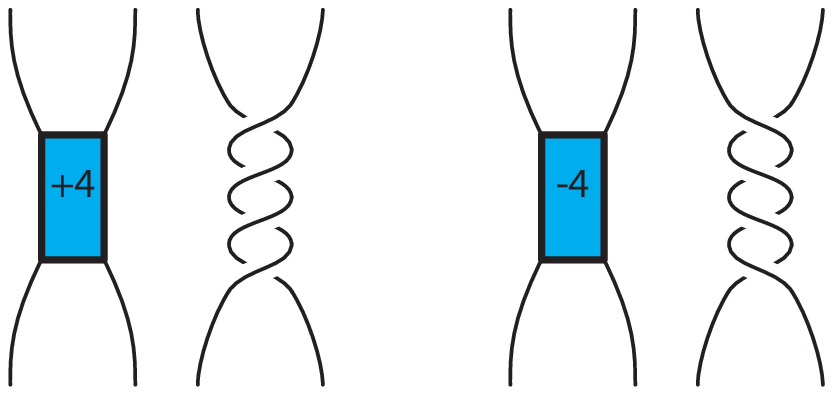}
\caption{}
\label{fig:twistboxes}
\end{figure}

{\bf (2) $L_1$ does not have tunnel number one.}\\
The arc $L_1$ in the final picture of Figure~\ref{fig:tangleisotopy} lifts to the knot $L_1$ in $M$.  The exterior of a small ball neighborhood of the arc $L_1$
is a ball that intersects the fixed set of the quotient in a tangle, $\omega$.
The double cover of this outside ball branched over $\omega$ is $M - \nbhd(L_1)$.  

When $|j_1| = 1$, the tangle $\omega$ is isotopic to the tangle $\tau$ of
Claim~\ref{clm:j1=1}, which shows that its double branched cover does 
not have tunnel number one. 

So assume $|j_1| > 1$ and that $M - \nbhd(L_1)$ has tunnel number one. 
The double branched cover of the link we get by adding
a rational tangle to $\omega$ is a 
Dehn filling of $M - \nbhd(L_1)$. Any such Dehn filling must have Heegaard
genus at most two. Consider the link, $l$, gotten by adding $n > 2$ vertical
twists to $\omega$. The resulting link is the union of $\tau$ 
(Figure~\ref{fig:tau})
and $\mu$ where $\mu$ is a non-rational tangle whose double branched cover
is a Seifert fiber space, $S$, over the disk with exceptional fibers of
order $|j_1|,n$. The Seifert fiber of $S$ is unique up to isotopy along
$\partial S$. By Claim~\ref{clm:j1=1}, the Heegaard genus $2$ manifold, 
$M_l$, that is the double branched
cover of $l$ is the union along an incompressible torus, $T$, of $N$
(the double branched cover of $\tau$) and $S$. As $N$ and $S$ are irreducible,
so is $M_l$. 

\medskip

{\bf Convention:} Below we will be considering the links in $S^3$ gotten
by filling the boundary sphere of tangles $\tau,\tau_1$ in 
Figures~\ref{fig:tau},
~\ref{fig:tau1} with rational tangles. Rational tangles are determined by
the slopes of their arcs on the bounding sphere with four marked points.
For an integer $n$, our convention will be that the rational tangles 
(or the corresponding slopes on
the sphere) are labeled: $n/1$ for two horizontal 
arcs with $n$ right-handed twists, $1/n$ for two vertical arcs with
$n$ left-handed twists. On the level of the double branched covers,
a slope on the tangle sphere determines a slope on the bounding torus
above, and a tangle filling results in a Dehn filling along that slope.

\medskip

First we show that $M_l$ cannot be a Seifert fiber space. 
Otherwise the separating torus $T$ would
have to be vertical in that Seifert fiber space and $N$ would admit
a Seifert fibration whose fiber agrees with the Seifert fiber of $S$.
But Dehn filling $N$ along this fiber gives $S^2 \times S^1$ (adding the
$1/0$-tangle to $\tau$ gives the unlink). This means that $N$ is the
circle bundle over the \mobius band. 
Thinking of $N$ as a Seifert fiber space over the disk with two exceptional
fibers, each of order $2$, we see that no Dehn filling of $N$ contains 
a separating essential torus. But filling $\tau$ with the $0/1$-tangle 
gives a link that is a union of two tangles $\tau_1, \tau_2$
whose double branched cover contains a separating incompressible
torus by Claim~\ref{clm:not2bridge} (and where the double branched cover of
$\tau_2$ is Seifert fibered over the disk with two exceptional fibers).

Thus $M_l$ has a non-trivial torus decomposition and  
\cite{kobayashi} describes such manifolds with Heegaard genus $2$.

\begin{claim}\label{clm:canonical}
Assume $M$ is closed, connected, irreducible, has Heegaard genus $2$, 
contains an essential torus, and is not a Seifert fiber space. 
If $M$ does not contain an 
essential non-separating torus, then the canonical torus decomposition of 
$M$ is as in $(i)-(iv)$ of \cite{kobayashi}. If $M$ contains an essential, 
non-separating torus, then the canonical torus decomposition of $M$ is as in 
$(v)$ of \cite{kobayashi}, with the exception that one of the decomposing 
tori is removed if at least one of $M_1$ or $M_2$ is a product $T^2 \times I$.

The canonical torus decomposition of $M$ has the property that any torus
is isotopic into one of the pieces of the decomposition.
\end{claim}

\begin{proof}
This is the content of the proof of the Main Theorem of \cite{kobayashi}.
When $M$ contains an essential, non-separating torus, \cite{kobayashi} 
shows that
$M$ has a decomposition as in $(v)$. In this case, the identification
described between the components of $\partial M_1$ and $\partial M_2$ 
guarantees that the decomposition is the canonical (minimal) torus 
decomposition, unless either $M_1$ or $M_2$ is $T^2 \times I$. In that
case by amalgamating $M_1$ and $M_2$ we get a torus decomposition which
must be minimal since $M$ is not a Seifert fiber space.

If $M$ does not contain a non-separating torus, then the proof of the 
Main Theorem of \cite{kobayashi} shows that a canonical decomposition of
$M$ is of one of the forms $(i)-(iv)$. 

A canonical torus decomposition (see for example \cite{hatcher:3mflds}) 
has the property that any torus is
isotopic into a piece of the decomposition (else there would be contigous
Seifert pieces where the fibers agree -- contradicting the minimality of 
the decomposition).
\end{proof} 

We now argue that $M_l=N \cup_T S$ does not have a canonical torus decomposition
of the form $(i)-(v)$ of \cite{kobayashi} guaranteed by 
Claim~\ref{clm:canonical},
thereby showing that $M_l$ cannot have Heegaard genus $2$.
Lemmas 4.2 and 4.4 of \cite{kobayashi}
show that the exterior of a two-bridge knot or link is atoroidal. Lemma 5.2
of \cite{kobayashi} shows that the exterior of a one-bridge knot in a 
lens space of class $L_K$ of the Main Theorem of \cite{kobayashi} is atoroidal
unless it is Seifert fibered over the \mobius band with a single exceptional
fiber. In this case, the unique incompressible torus which is not boundary
parallel is a vertical torus which bounds the neighborhood of a vertical
Klein bottle. Finally, note that since $|j_1|,n>2$, $S$ is not the 
exterior of a two-bridge knot in $S^3$. 
\begin{itemize}
\item[(i),(ii).] Assume there is a canonical decomposition as in 
$(i)$ or $(ii)$ of \cite{kobayashi}.
$T$ is not isotopic to $\partial M_1$ since $N$ does not have tunnel number
one by Claim~\ref{clm:j1=1} and since $S$ is not the exterior of a 
two-bridge knot.
In a decomposition as in $(i)$, since $M_1$ is atoroidal, $T$ must be
an essential torus in $M_2$. By Lemma 5.2 of \cite{kobayashi},  
$M_2$ is Seifert fibered over a \mobius band with a single exceptional
fiber, and $T$ bounds the neighborhood of a vertical Klein bottle in $M_2$. 
But this contradicts that $S$ is atoroidal and Seifert fibered over the
disk (with a unique fibration). In a decomposition as in $(ii)$, $M_1$ must be
Seifert fibered over the \mobius band with at least one exceptional fiber, and
$T$ must be vertical in this fibration. As $S$ is atoroidal, it must be the
side of $T$ that lies in $M_1$. Then $N$ is the union of a circle bundle
over a once-punctured \mobius band and $M_2$ -- where the circle fiber 
is identified with the meridian of $M_2$. This implies that $N$ has 
tunnel number
one (one can find a tunnel for a two-bridge knot where two meridians 
represent jointly primitive curves in the genus $2$ handlebody). But
this contradicts Claim~\ref{clm:j1=1}. 

\item[(iii).] Assume $M_l$ has a canonical decomposition as in $(iii)$. 
If $T$ were 
isotopic to $\partial M_1$, then $M_2$  would have to be $N$ ($S$ is not the
exterior of a two-bridge knot in $S^3$). But $N$ is not tunnel number one. Thus 
$M_1$ must be Seifert fibered over the
disk with three exceptional fibers and $T$ must must be a vertical torus
in $M_1$. Thus one side of $T$ is the union of a Seifert fiber space
over the annulus with one exceptional fiber and $M_2$ -- where the
Seifert fiber is identified with the meridian of the two-bridge knot
exterior $M_2$. But such a manifold has tunnel number one (a meridian
is primitive in the tunnel one handlebody of a two-bridge knot exterior).
That is, both sides of $T$ have tunnel number one, contradicting 
Claim~\ref{clm:j1=1}.
\item[(iv).]  Assume that $M_l$ is decomposed as in $(iv)$ into the
three atoroidal manifolds $M_1,M_2,M_3$. Then
$T$ is isotopic to a component of $\partial M_3$. But each side of
$T$ has tunnel number one (e.g. $M_3 \cup M_2$ has tunnel number one,
since the Seifert fiber of $M_2$ is identified with the meridian of $M_3$
which is primitive in its tunnel one handlebody). This contradicts
Claim~\ref{clm:j1=1}.
\item[(v).] Assume $M_l$ is decomposed into $M_1$ and $M_2$ as in $(v)$.
The separating torus $T$ must be a 
vertical torus in $M_1$ where $M_1$ is Seifert fibered over an annulus
with two exceptional fibers. Thus both sides of $T$ have tunnel number
one, contradicting Claim~\ref{clm:j1=1} (note that the union of $M_2$
with the Seifert fiber space over the $3$-punctured sphere has tunnel
number one, since the Seifert fibers are identified with meridians
of the two-bridge link -- which are jointly primitive in its tunnel
one handlebody).
\end{itemize}


This shows that $M_l$ does not have Heegaard genus $2$, which contradicts 
our assumption that $M-\nbhd(L_1)$ has tunnel number one, once we verify the 
supporting claims.

\begin{figure}
\includegraphics[width=1in]{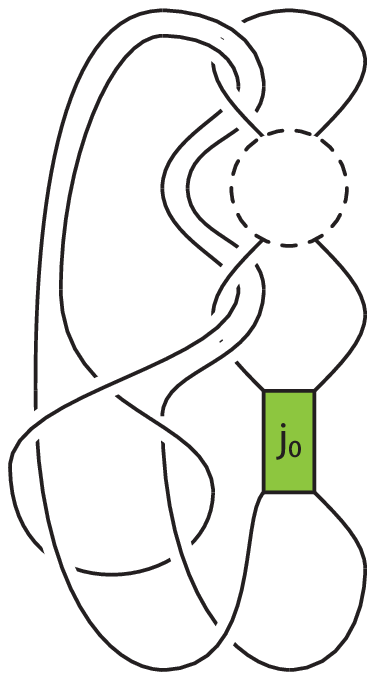}
\caption{}
\label{fig:tau}
\end{figure}

\begin{claim}\label{clm:j1=1}
Let $\tau$ be the tangle pictured in Figure~\ref{fig:tau}. The double branched cover, $N$, of $\tau$ is irreducible, $\bdry$-irreducible, and has tunnel number greater than one. 
\end{claim}

\begin{figure}
\includegraphics[width=\textwidth]{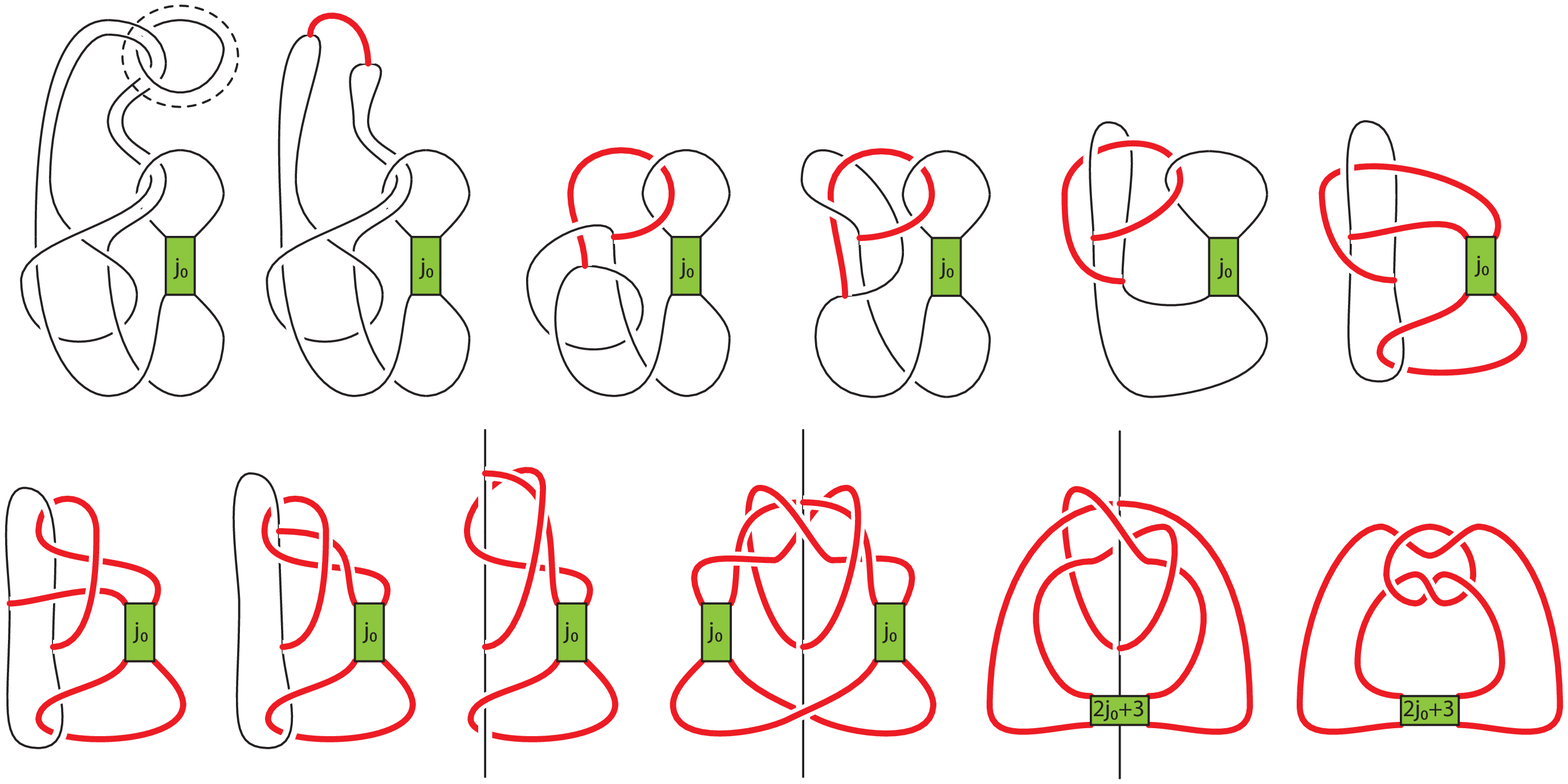}
\caption{}
\label{fig:extL1}
\end{figure}

\begin{figure}
\includegraphics[width=1in]{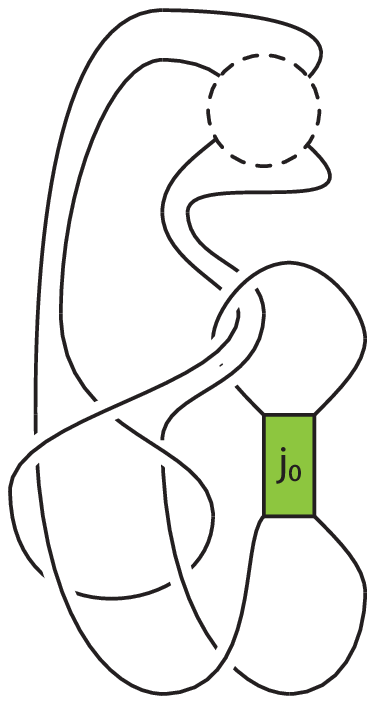}
\caption{}
\label{fig:tau1}
\end{figure}

\begin{proof}
Filling $\tau$ with the $1/0$-tangle gives the two component unlink. Thus the corresponding Dehn filling of $N$ is $S^2 \times S^1$.

The $0/1$-filling of $\tau$ gives the three component link, $l'$, in the first picture of Figure~\ref{fig:extL1}.  That picture shows this link decomposed along a Conway sphere into two tangles $\tau_1,\tau_2$, where $\tau_1$ is the tangle of Figure~\ref{fig:tau1} and where $\tau_2$ is the tangle gotten by taking a horizontal, unknotted circle along with 2 parallel, vertical arcs linking the circle. The double branched cover of $S^3$ along $l'$, $M_{l'}$, is the union along a torus, $T'$ of the double branched cover of $\tau_1$, $N_1$, and the double branched cover of $\tau_2$, $N_2$.  $N_2$ is a twisted $I$-bundle over a Klein bottle which has exactly two descriptions as a Seifert fiber space corresponding to
two different slopes on $\partial N_2$: one as Seifert-fibered over a disk with two exceptional fibers of order $2$, the other as a circle bundle over a \mobius band. By Claim~\ref{clm:not2bridge}, $N_1$ is irreducible, $\partial$-irreducible, and atoroidal. Thus $T'$ is an incompressible torus in $M_{l'}$, and either $M_{l'}$ is a toroidal Seifert fiber space or $N_1 \cup N_2$ is a non-trivial, canonical torus decomposition of $M_{l'}$. 

From the above discussion, the $1/0$-Dehn filling of $N$ is prime and the 
$0/1$-Dehn filling of $N$ is irreducible, thus $N$ is irreducible.  As the filling $M_{l'}$ contains an essential torus, $N$ is not a solid torus. Thus $N$ is irreducible and $\partial$-irreducible.

Assume for contradiction that $N$ has tunnel number one. Then the toroidal $M_{l'}$ has Heegaard genus two. 

First we show that $M_{l'}$ is not a Seifert fiber space. Suppose it is. Then $N_1$ is a Seifert fiber space whose fiber is a fiber of one of the two Seifert fibrations of $N_2$, which is either the $1/0$-slope (as a Seifert fiber space over the disk) or the $0/1$-slope (as the Seifert fiber space over the \mobius band). However, the $0/1$-filling of $\tau_1$ is the unknot, implying that $N_1$ is the exterior of a knot in $S^3$ whose meridian is this $0/1$-slope. But the Seifert fiber of a knot exterior in $S^3$ is never meridianal. Thus it must be that $N_1$ is a Seifert fibered knot exterior in $S^3$ whose fiber has slope $1/0$.  Thus $M_{l'}$ is a Seifert fibered space over the $2$-sphere with four exceptional fibers, two of which have orders greater than $2$ (by Claim~\ref{clm:not2bridge} $N_1$ is not the exterior of a two-bridge knot). However, such a Seifert fibered space cannot have Heegaard genus $2$.  By \cite{moriah-schultens}, the splitting would have to be horizontal or vertical. It cannot be vertical because there are too many exceptional fibers. It cannot be horizontal by \cite{sedgwick}. Thus $M_{l'}$ is not a Seifert fiber space.

Thus $N_1 \cup N_2$ is a non-trivial canonical decomposition of the genus $2$ manifold $M_{l'}$. In particular, any torus in $M_{l'}$ is isotopic to $T'$. The main theorem of \cite{kobayashi} describes the possible canonical decompositions of $M_{l'}$ (see 
Claim~\ref{clm:canonical}). We rule out each of these possibilities.  ($M_1$, $M_2$, $M_3$ follow the notation of \cite{kobayashi}.)

\begin{itemize}
\item[(i).] We rule out conclusion $(i)$. By Claim~\ref{clm:not2bridge}, $N_1$ is not a Seifert fibered space unless $j_0=0$ (since $|j_0| \neq 1,2$ by hypothesis).  In that case the slope of the regular fiber of $N_1$ on $T'$ is $-1/2$, but this is not the slope of a lens space filling of $N_2$.
 Thus it must be that $M_1$ is $N_2$ and $M_2$ is $N_1$. That is, $N_1$ is the exterior of a one-bridge knot in a lens space whose meridian is identified along $T'$ with the Seifert fiber, with slope $1/0$, of $N_2$ coming from its fibration over the disk. Filling $N_1$ along this meridian is a lens space. However this filling is the double branched cover of the two component link gotten by filling $\tau_1$ with the $1/0$-tangle. By Hodgson-Rubinstein \cite{hr}, this link must be a two-bridge link. However, this is impossible since, for $j_0 \neq 1,2$, one component of this link is knotted.

\item[(ii).]  Conclusion $(ii)$ does not hold since neither $N_1$ nor $N_2$ is a two-bridge knot exterior.  (Claim~\ref{clm:not2bridge} for $N_1$ and the fact that $N_2$ contains a Klein bottle.)

\item[(iii).] Conclusion $(iii)$ does not hold as neither $N_1$ nor $N_2$ is a two-bridge knot exterior.

\item[(iv).] Conclusion $(iv)$ does not hold as the torus decomposition of $M_{l'}$ has only two pieces.

\item[(v).] Conclusion $(v)$ does not hold as $M_{l'}$ does not contain a non-separating torus.
\end{itemize}

This finishes the proof of Claim~\ref{clm:j1=1}.
\end{proof}

\begin{claim}\label{clm:not2bridge}
Let $\tau_1$ be the tangle in Figure~\ref{fig:tau1}. Let $N_1$ be the double cover of the tangle ball branched over $\tau_1$. Then $N_1$ is the exterior of the $(-2,3,2j_0+3)$-pretzel knot in $S^3$ pictured in the final picture of Figure~\ref{fig:extL1}.  $N_1$ is irreducible, $\partial$-irreducible, and atoroidal for every $j_0$.  $N_1$ is Seifert fibered only when $j_0=1,0,-1$, and when $j_0=0$ the Seifert fiber on $\bdry N_1$ has slope $-1/2$. $N_1$ is the exterior of a  two-bridge knot only when $j_0=-1,-2$.
\end{claim}


\begin{proof}
Figure~\ref{fig:extL1} is a sequence of pictures identifying $N_1$ as the exterior of the non-trivial $(-2,3,2j_0+3)$-pretzel knot.  Since pretzel knots are not satellite knots \cite{kimlee}, $N_1$ is irreducible, $\bdry$-irreducible, and atoroidal for all $j_0$.  These $(-2,3,2j_0+3)$-pretzel knots are the $(3,5)$, $(3,4)$, and $(2,5)$-torus knots when $j_0$ is $1,0,-1$ respectively and hyperbolic otherwise \cite{kimlee}. Moreover, they are two-bridge knots only when $j_0=-1,-2$, e.g.\ \cite{kawauchi}.  When $j_0=0$, the slope of the Seifert fiber on $\bdry N_1$ may be determined as the slope of the tangle that fills $\tau_1$ of Figure~\ref{fig:tau1} to produce the connected sum of the $(3,1)$ and $(-4,1)$ torus knots.
\end{proof}

{\bf (4) Filling $M-\nbhd(L_1)$ along the slope $\bdry \hatR$ is atoroidal.}\\
The $R$--framing of $L_1 \subset S^3$ is the $\hatR$--framing of $L_1 \subset M$ and corresponds to the page framing of the arc $L_1$ at the end of Figure~\ref{fig:tangleisotopy}.   Therefore the $\hatR$--framed surgery on $L_1 \subset M$ is the double branched cover of the link in the second picture of Figure~\ref{fig:tangleatoroidal}.  The subsequent pictures show isotopies of this link to the split link comprised of the unknot and the $(2,j_1)$--torus link.  Thus the $\hatR$--framed surgery on $L_1 \subset M$ is homeomorphic to $S^1 \times S^2 \# L(j_1,1)$ which is atoroidal.

\begin{figure}
\includegraphics[width=5in]{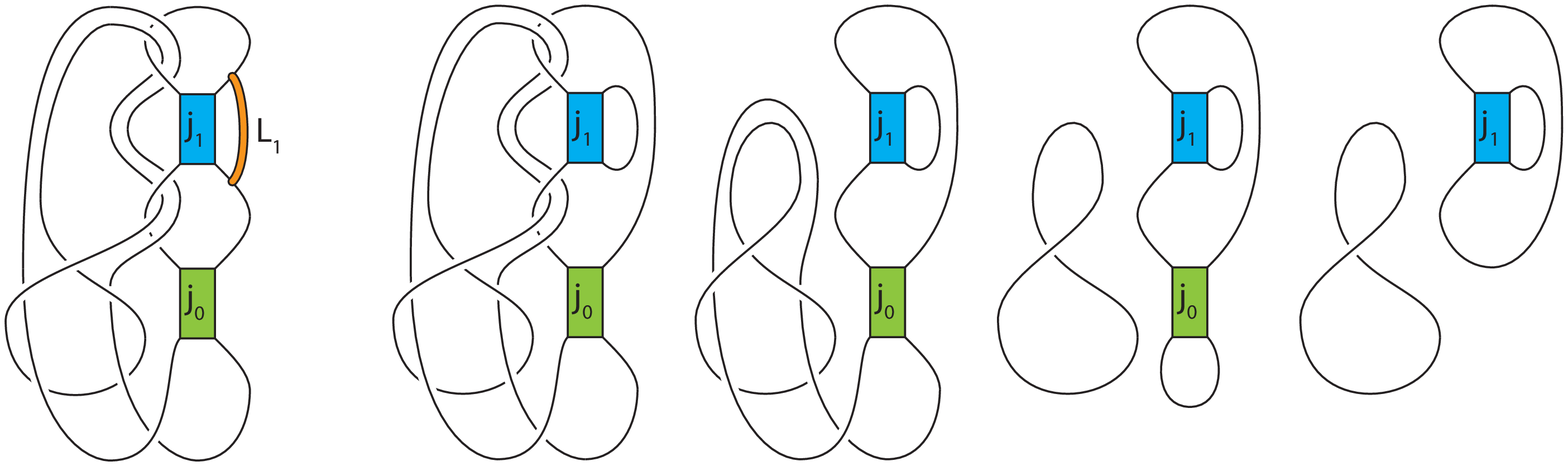}
\caption{}
\label{fig:tangleatoroidal}
\end{figure}


{\bf (3) $L_1$ is not the cable of a tunnel number $1$ knot in $M$, where the cabling annulus has the same slope on $L_1$ as $\bdry \hatR$.}
Suppose $L_1$ is cabled as described about the tunnel number one knot $J$ in $M$. Let $A_1$ be the cabling annulus, properly embedded in $M-\nbhd(L_1)$, whose 
two boundary components have the same slope on $L_1$ as $\bdry \hatR$ (the
$0$-slope).

First assume $|j_1|=1$.    
As $L_1$ is cabled, the $0$--surgery in $M$ along $L_1$ will produce a manifold with a lens space (of positive, finite order in first homology) summand. But we saw above that $0$--surgery on $L_1$ produces $S^1 \times S^2 \# L(|j_1|,1)$ which is $S^1 \times S^2$.

Thus we may assume that $|j_1| > 1$.  As $|j_0|,|j_1| \geq \nu$, 
$M$ is hyperbolic.  Because $L_1$ is isotopic to a meridian of $J_1$ (see Figure~\ref{fig:simplifyingisotopy}), $L_1$ is also a $(j_1,1)$--cable of $J_1'$, the dual to the $-1/j_1$--surgery on $J_1$. Furthermore, the slope of this cabling annulus $A_2$ on $L_1$ is that of $\bdry \hatR$. 
Now the exterior of $J_1'$ in $M$ is $N$, the double branched cover of $\tau$
in Claim~\ref{clm:j1=1}. By that Claim, $J_1'$ does not have tunnel number one. 
Thus $J$ is not isotopic to $J_1'$ in $M$.

We may isotop $A_1,A_2$ in $M-\nbhd(L_1)$ so that they intersect in parallel
essential curves in the interiors of $A_1$ and $A_2$. Let $V_i$ be the cabling solid torus for $A_i$ in the exterior of $L_1$. Note that any incompressible annulus in $V_1$ is $\bdry$--parallel. Hence for $\{i,j\}=\{1,2\}$, we may assume that $\partial A_i$ lies outside of $V_j$, and that each component of $A_i \cap V_j$ is parallel to $\bdry \nbhd(L_1) \cap V_j$. First assume that $A_1 \cap A_2$ is non-empty. Then $V_j - \nbhd(A_i)$ consists of solid tori, exactly one of which, $\calC_j$, has the property that $A_i \cap \calC_j$ is not longitudinal. The core of $\calC_j$ is isotopic to the core of $V_j$. If $\calC_j$ lies in $V_i$ then $\calC_j=\calC_i$ and $J$ would be isotopic to $J_1'$. So it must be that $\calC_j$ lies outside of $V_i$. As $\calC_j$ meets $V_i$ in a subannulus of $A_i$, $\calT = \calC_j \cup V_i$ is a Seifert fiber space over the disk
with two exceptional fibers. Note that $L_1$ is isotopic to a regular fiber of $\calT$. If the boundary of $\calT$ is compressible in $M$, then either it, and hence $L_1$, is contained in a ball, or $M$ is either a small Seifert fiber space or the connected sum of lens spaces. If its boundary is incompressible, then $M$ is toroidal. As $M$ is hyperbolic, it must be that $L_1$ lies in a ball in $M$. Then $M$ is a connect summand of the $0$--surgery on $L_1$. But $M$ is hyperbolic, and $0$--surgery on $L_1$ produces $S^2 \times S^1 \# L(|j_1|,1)$.

Thus it must be that $A_1, A_2$ are disjoint. Then $V_1 \cup \nbhd(L_1) \cup V_2$ is a Seifert fiber space over the disk with two exceptional fibers, of which $L_1$ is a regular fiber. As above, this contradicts the hyperbolicity of $M$.
\end{proof}

Teragaito also describes the link $\calL'_{1,1}$ in \cite{teragaito} and states that $0$--surgery ($+4$--surgery with respect to the Seifert framing) on each of the knots $K'^n_{1,1}$ yields the same Seifert fibered manifold of type $S^2(3,4,8)$.  We observe this as follows:  Continuing from Figure~\ref{fig:tangleisotopy}, Figure~\ref{fig:teragaito2nd} shows that setting $j_0=j_1=1$ produces a link isotopic to the Montesinos link $m(0;-1/3, 5/8,-1/4)$.  The double branched cover of this link is the manifold $M_{1,1}$ that results from the $0$--surgery on $K'^0_{1,1}$ and is a Seifert fibered manifold of the type claimed.  

\begin{thm}\label{thm:secondfamily}
For the second Teragaito family, where $j_0=1$ and $j_1=1$,  $\{\bn_2(K^n)\}$ is
finite.
\end{thm}
\begin{proof}
We set $j_0=1$ and $j_1=1$. Figure~\ref{fig:teragaito2nd} also keeps track of the orange arc that lifts to $L_1$.  The final link of this figure is decomposed in Figure~\ref{fig:teragaito2ndTN1} into two $3$--strand trivial tangles, one of which contains the orange arc as a ``core arc''.  That is, in the genus $2$ handlebody that is the double branched cover of this $3$--strand trivial tangle, the orange arc lifts to a core.  Consequently, this implies $L_1$ is a core curve of a genus $2$ splitting of $M$.  Hence the tunnel number of $L_1 \subset M$ is one.  Moreover, if $\hat{F}$ is a Heegaard surface of this splitting, then $L_1$ may be isotoped into $\hat{F}$ with any desired framing.  Therefore there is an isotopy of the annulus $\hat{R}$ into $\hat{F}$. As argued at the end
of the proof of Corollary~\ref{cor:iff}, there is an upper bound for $\bn_{\hatF}(K^n)$ and hence for $\bn_2(K^n)$ as well.
\end{proof}

\begin{figure}
\centering\
\includegraphics[width=\textwidth]{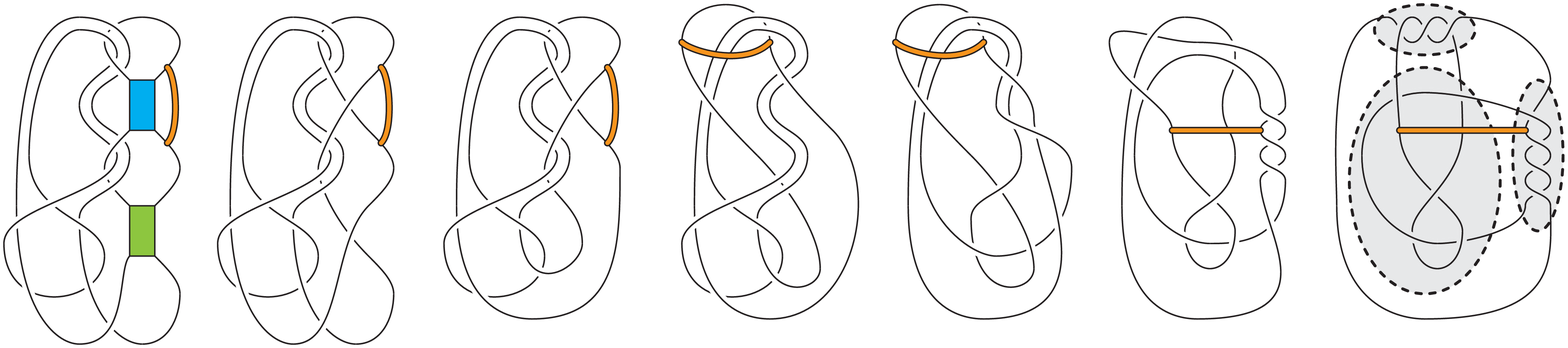}
\caption{}
\label{fig:teragaito2nd}
\end{figure}

\begin{figure}
\centering
\includegraphics[width=\textwidth]{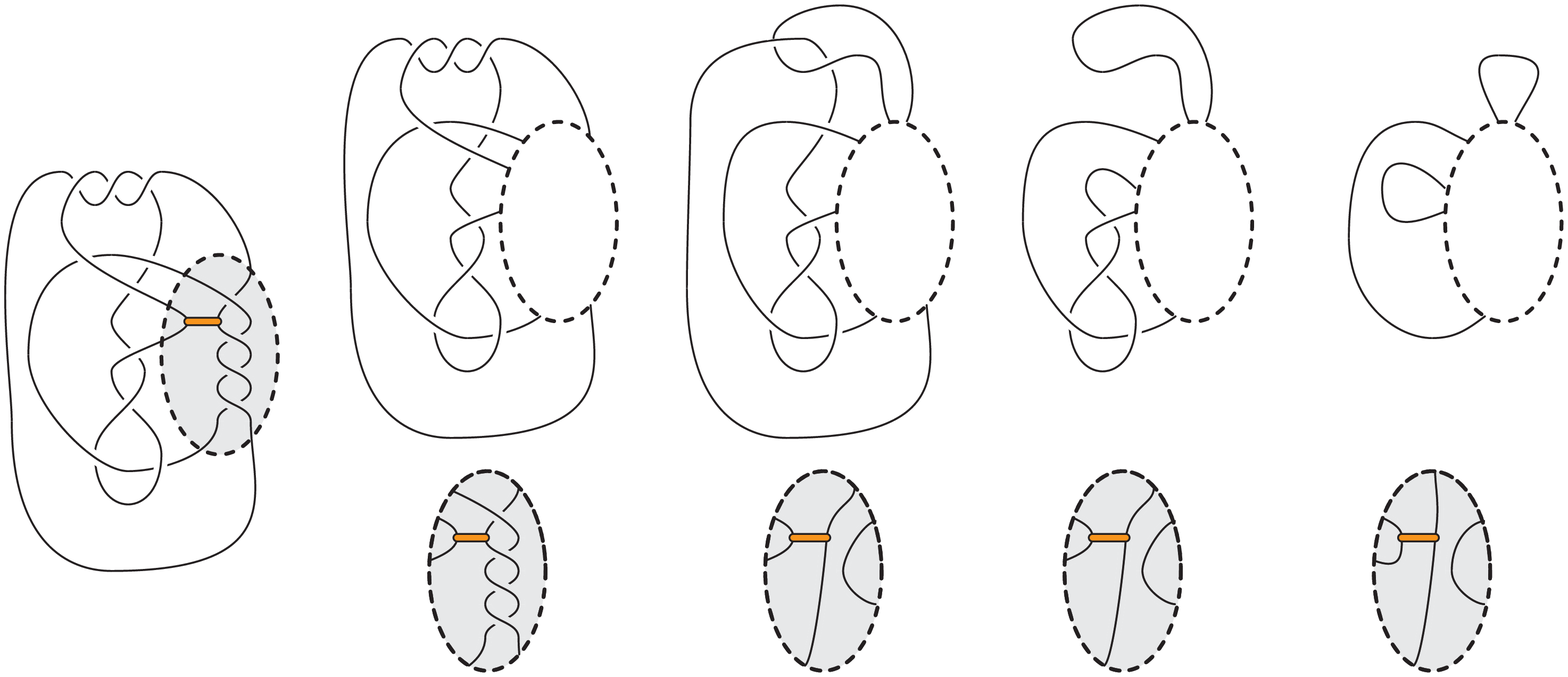}
\caption{}
\label{fig:teragaito2ndTN1}
\end{figure}

\begin{defn}
A non-orientable, closed surface with Euler characteristic $-1$ is called
a {\em Dyck's surface}.
\end{defn}

\begin{lemma}\label{lem:noDycks}
For infinitely many pairs $(j_0,j_1)$ the manifold $M_{j_0,j_1}$ 
is hyperbolic and contains no Dyck's surface.
\end{lemma}

\begin{proof}
Recall that $M_{j_0,j_1}$ is obtained by Dehn surgery on the link 
$K' \cup J_0 \cup J_1$ in $S^3$: $0,-1/j_0,-1/j_1$-surgery on 
$K',J_0,J_1$ respectively. Let $X$ be exterior of $K' \cup J_0 \cup J_1$
in $S^3$ and $T,T_0,T_1$ be the components of $\partial X$ corresponding
to $K',J_0,J_1$ respectively. Recall that $W_0$ is the $0$-filling of 
$X$ along $T$ and that the interiors of both $X$ and $W_0$ are hyperbolic of finite volume.

First we note that $M_{0,1}$ contains no Dyck's surface, Klein bottle, or
projective plane by the appendix in \cite{bgl:og2hsfs} ($M_{0,1}$ is the
Seifert fiber space $S^2(-1/2,1/6,2/7)$). As $M_{0,1}$ is a filling
of $W_0$, this implies that $W_{0}$ contains no Dyck's surface.  

\begin{claim}\label{clm:dycksj1}
For $|j_1| > 108$, the manifold $W_0(j_1)$ obtained by 
$-1/j_1$-filling the 
$T_1$ boundary of $W_0$ contains no Dyck's surface. 
\end{claim}
\begin{proof}
Assume $|j_1| > 108$.
The obvious disk that $K'$ bounds in the final picture of 
Figure~\ref{fig:simplifyingisotopy} gives rise in $W_0$ 
to a $4$-punctured disk with punctures on $T_1$. By tubing an appropriate
pair of these punctures we get a $2$-punctured torus, $Q'$, properly embedded
in $W_0$ whose two boundary components are coherently oriented curves
representing meridians (slope $1/0$) of $J_1$. Assume for contradiction
there is Dyck's surface $S$ in $W_0(j_1)$. Let $J_1'$ be the core of the
attached solid torus at $T_1$ in $W_0(j_1)$. Isotop $S$ in $W_0(j_1)$
to intersect $J_1'$ minimally. Let $S'=W_0 \cap S$. As $W_0$ contains no
Dyck's surface, $\partial S'$ is
a non-empty collection of curves of slope $-1/j_0$ on $T_1$.  
Isotop $\partial Q', \partial S'$ to intersect minimally in $W_0$. 
Then no arc of $Q' \cap S'$ is boundary parallel in either $Q'$ or $S'$
(note that a boundary parallel arc in $Q'$ is orientation-preserving so the
Parity Rule still applies). 
Let $A$ be the punctured genus $2$ surface coming from a regular 
neighborhood of $S'$ in $W_0(j_1)$. Then no arc of $Q' \cap A$ 
is boundary parallel in $Q'$ or in $A$. Consider the graphs of intersection $G_A,G_{Q'}$ coming from the arcs of $Q' \cap A$. Then $G_A,G_{Q'}$ have no monogons. 

As the distance between the slopes of the boundaries of $Q'$ and $A$ on $T_1$
is $|j_1|$, $G_{Q'}$ has $|j_1||\partial A|$ edges. Let $\widetilde{G_{Q'}}$ be
the reduced graph of $G_{Q'}$ gotten by amalgamating parallel edges of 
$G_{Q'}$. The proof of Claim~\ref{claim:paralleledges1} shows that 
$\widetilde{G_{Q'}}$ has
at most $6$ edges. Thus $G_{Q'}$ 
must have a collection, $\mathcal{E}$, of at least $|j_1||\partial A|/6$ parallel 
edges. Let $G_A(\mathcal{E})$ be the subgraph of $G_A$ corresponding to these
edges (along with all vertices of $G_A$). Then the valence of each vertex of 
$G_A$ is at least $|j_1|/6>18$. By Claim~\ref{claim:paralleledges2}, two of the edges of
$G_A(\mathcal{E})$ are parallel on $G_A$. As in the proof of 
Lemma~\ref{lem:largengivesnonhyp}, the
union of the disks bounded in $Q',A$ by an innermost pair of edges, gives 
rise to a \mobius band properly embedded in $W_0$. But this contradicts
the hyperbolicity of $W_0$.
\end{proof} 

\begin{claim}\label{clm:dycksj0}
Assume $W_0(j_1)$ is hyperbolic. If $|j_1| > 108$ and $M_{j_0,j_1}$ and $M_{j_0',j_1}$ both contain
Dyck's surfaces, then $|j_0-j_0'| \leq 324$.
\end{claim}

\begin{proof}
Assume for contradiction that $|j_1| > 108$, $|j_0-j_0'| > 324$, and 
that $S \subset M_{j_0,j_1}$ and $F \subset M_{j_0',j_1}$ are
embedded Dyck's surfaces. Isotope $S,F$ so that they
intersect the core of the solid torus attached to $T_0$ in $M_{j_0,j_1},
M_{j_0',j_1}$ minimally. Let $S'=S \cap W_0(j_1), F'=F \cap W_0(j_1)$.
By Claim~\ref{clm:dycksj1}, $\partial S'$ is a non-empty 
collection of curves of slope $-1/j_0$
and $\partial F'$ is a non-empty collection of curves of slope $-1/j_1$
on $T_2$ in $\partial W_0(j_1)$. Isotope $S',F'$ to intersect minimally.
Then no arc of $S' \cap F'$ is boundary parallel in either $S'$ or $F'$.
Let $A,B$ be the boundary of a regular neighborhood of $S',F'$ (resp.)
in $W_0(j_1)$. Then $A$ and $B$ are both punctured surfaces of genus
$2$. Consider the graphs of intersection $G_A,G_B$. Neither $G_A$ nor
$G_B$ have monogons. The valence of each vertex of $G_A$ is 
$|j_0-j_0'| |\partial B|$. Let $\widetilde{G_A}$ be the reduced graph
of $G_A$. By Claim~\ref{claim:paralleledges2}, 
the valence of some vertex of $\widetilde{G_A}$ is at most $18$. 
This implies that there must be a group of parallel edges $\mathcal{E}$ in
$G_A$ with cardinality $|j_0 - j_0'| |\partial B|/18$. Let $G_B(\mathcal{E})$
be the subgraph of $G_B$ gotten from the edges corresponding to $\mathcal{E}$.
The vertices of $G_B(\mathcal{E})$ have valence $|j_0-j_0'|/18 > 18$. Thus
again Claim~\ref{claim:paralleledges2}, implies that two edges of 
$G_B(\mathcal{E})$ are 
parallel in $G_B$. Once again, the disks bounded on $A,B$ by an innermost 
pair of such edges, gives rise to a \mobius band properly embedded in 
$W_0(j_1)$. But
this contradicts that $W_0(j_1)$ is hyperbolic.
\end{proof}  

Recall that as long as $|j_0|,|j_1| \geq \nu$, $M_{j_0,j_1}$ is hyperbolic.
As $W_0$ is hyperbolic, there is constant $\nu'$ such that if $|j_1|>\nu'$
then $W_0(j_1)$ is hyperbolic. Fix $|j_1| > \max\{\nu,\nu',108\}$. By 
Claim~\ref{clm:dycksj1}, for all but finitely values of $j_0$, $M_{j_0,j_1}$
will be hyperbolic and contain no Dyck's surface.
\end{proof}

\providecommand{\bysame}{\leavevmode\hbox to3em{\hrulefill}\thinspace}
\providecommand{\MR}{\relax\ifhmode\unskip\space\fi MR }
\providecommand{\MRhref}[2]{%
  \href{http://www.ams.org/mathscinet-getitem?mr=#1}{#2}
}
\providecommand{\href}[2]{#2}


\begin{thebibliography}{10}
\bibliographystyle{amsplain}


\bibitem{bgl:bnaids}
Kenneth~L. Baker, Cameron~McA. Gordon, and John Luecke, \emph{Bridge number, Heegaard genus, and non-longitudinal Dehn surgery}, preprint.

\bibitem{bgl:og2hsfs}
Kenneth~L. Baker, Cameron~McA. Gordon, and John Luecke, \emph{Obtaining genus
 $2$ Heegaard splittings from surgery}, to appear Algebraic and Geometric 
Topology.


\bibitem{snappy}
M.\ Culler, N.\ M.\ Dunfield, and J.\ R.\ Weeks, \emph{SnapPy, a computer program for studying the geometry and topology of $3$--manifolds}, {\tt http://snappy.computop.org}

\bibitem{cgls:dsok}
Marc Culler, Cameron~McA. Gordon, John Luecke, and Peter~B. Shalen, \emph{Dehn
  surgery on knots}, Bull. Amer. Math. Soc. (N.S.) \textbf{13} (1985), no.~1,
  43--45. \MR{MR788388 (86k:57013)}


  
\bibitem{gabai:fatto3mIII}
David Gabai, \emph{Foliations and the topology of {$3$}--manifolds. {III}}, J.
  Differential Geom. \textbf{26} (1987), no.~3, 479--536. \MR{MR910018
  (89a:57014b)}

%

\bibitem{gordon}
Cameron~McA. Gordon, \emph {Boundary slopes of punctured tori in {$3$}-manifolds}, Trans. Amer. Math. Soc. \textbf{350} (1998), no.~5, 1713--1790. 


\bibitem{GLi}
Cameron~McA. Gordon and R.A.~Litherland, \emph{Incompressible planar surfaces in {$3$}--manifolds}, Topology Appl. \textbf{18} (1984), no.~2-3, 121--144. \MR{MR769286 (86e:57013)}

\bibitem{hatcher:3mflds}
Allen Hatcher, \emph{Notes on Basic $3$--manifold Topology}, unpublished, {\tt http://www.math.cornell.edu/$\sim$hatcher/}

\bibitem{hr}
C. Hodgson and H. Rubinstein, \emph{Involutions and isotopies of lens spaces}, In \emph{Knot theory and manifolds}, Lecture Note in Math. vol. 1144 pp. 60--96 (Springer-Verlag 1983).

\bibitem{kawauchi}
Akio Kawauchi, \emph{A survey of knot theory}, (Birkh\"auser Verlag, Berlin, 1996).

\bibitem{kimlee}
Dongseok Kim and Jaeun Lee, \emph{Some invariants of pretzel links}, Bull. Austral. Math. Soc. \textbf{75} (2007) no.~2, 253--271.


\bibitem{kobayashi}
Tsuyoshi Kobayashi, \emph{Structures of the Haken manifolds with Heegaard splittings
of genus two}, Osaka J. Math. \textbf{21} (1984), 437-455.

\bibitem{mmm}
Thomas Mattman, Katura Miyazaki, and Kimihiko Motegi, \emph{Seifert-fibered surgeries which do not arise from primitive/Seifert-fibered constructions}, Trans. Amer. Math. Soc. \textbf{358} (2006), no.~9, 4045--4055.

\bibitem{moriah-schultens}
Yoav Moriah and Jennifer Schultens, \emph{Irreducible Heegaard splittings of Seifert fiber spaces are 
horizontal or vertical}, Topology \textbf{37} (5) (1998), 1089-1112

\bibitem{osoinach}
John K.\ Osoinach Jr., \emph{Manifolds obtained by surgery on an infinite number of knots 
in {$S^3$}}, Topology \textbf{45} (2006), 725-733.

\bibitem{scharlemann}
Martin Scharlemann, \emph{Sutured manifolds and generalized {T}hurston norms}, J. Differential Geom. \textbf{29} (1989), no.~3, 557--614

\bibitem{schultens}
Jennifer Schultens, \emph{Weakly reducible Heegaard splittings of Seifert fibered spaces}, Top. Appl.
\textbf{100} (2000), 219-222

\bibitem{sedgwick}
Eric Sedgwick, \emph{The irreducibility of Heegaard splittings of Seifert fiber spaces}, 
Pac. Jour. Math. \textbf{190} (1999), no.~1, 173--199



\bibitem{teragaito}
M. Teragaito, \emph{A Seifert fibered manifold with infinitely many 
knot-surgery descriptions}, 
Int Math Res Notices (2007) Vol. 2007 doi:10.1093/imrn/rnm028  

%



\end{thebibliography}
\end{document}